\documentclass[pdflatex,sn-mathphys]{sn-jnl}

\usepackage{algorithm}
\usepackage{amssymb}
\usepackage{amsmath}
\usepackage{hyperref}
\usepackage{graphicx}

\usepackage{tabularx}
\usepackage{booktabs}
\usepackage{subfig}
\usepackage{graphicx}

\usepackage{caption}
\jyear{2022}%

\theoremstyle{thmstyleone}%
\newtheorem{theorem}{Theorem}
\newtheorem{proposition}[theorem]{Proposition}%
\newtheorem{lemma}[theorem]{Lemma}%

\theoremstyle{thmstyletwo}%
\newtheorem{remark}[theorem]{Remark}%
\newtheorem{assumption}{Assumption}

\newcommand{\wh}{\widehat}
\newcommand{\wt}{\widetilde}
\newcommand{\bb}{\mathbf b}
\newcommand{\bc}{\mathbf c}
\newcommand{\be}{\mathbf e}
\newcommand{\bg}{\mathbf g}
\newcommand{\bs}{\mathbf s}
\newcommand{\bu}{\mathbf u}
\newcommand{\bv}{\mathbf v}
\newcommand{\bx}{\mathbf x}
\newcommand{\by}{\mathbf y}
\newcommand{\bz}{\mathbf z}
\newcommand{\zero}{\mathbf 0}
\newcommand{\step}[2]{\beta_{#2}^{\rm{#1}}}
\newcommand{\invstep}[2]{\alpha_{#2}^{\rm{#1}}} 
\newcommand{\alphaa}{\alpha^{\rm{BB1}}}
\newcommand{\alphab}{\alpha^{\rm{BB2}}}
\newcommand{\altstep}{\widehat{\beta}} 
\newcommand{\betaa}{\beta^{\rm{BB1}}}
\newcommand{\betab}{\beta^{\rm{BB2}}}
\newcommand{\xia}{\xi_{\rm{low}}}
\newcommand{\xib}{\xi^{\rm{up}}}
\newcommand{\toneconst}{\rho}
\newcommand{\constone}{c_1} 
\newcommand{\consttwo}{c_2} 
\newcommand{\constray}{c} 
\newcommand{\gradcomp}[2]{\gamma_{#1}^{#2}}
\newcommand{\eps}{\varepsilon}
\newcommand{\ds}{\displaystyle}
\newcommand{\calo}{{\mathcal O}}
\newcommand{\calu}{{\mathcal U}}
\newcommand{\tol}{{\sf tol}}
\newcommand{\ph}{\phantom}
\DeclareMathOperator*{\argmin}{\arg\!\min}

\newcommand{\gf}[1]{#1}

\usepackage[normalem]{ulem}

\begin{document}

\title[A harmonic framework for stepsize selection in gradient methods]{A harmonic framework for stepsize selection \\ in gradient methods}


\author[1]{\fnm{Giulia} \sur{Ferrandi}}\email{g.ferrandi@tue.nl}

\author[1]{\fnm{Michiel E.} \sur{Hochstenbach}}\email{m.e.hochstenbach@tue.nl}

\author[2]{\fnm{Nata\v{s}a} \sur{Kreji\'c}}\email{natasak@uns.ac.rs}

\affil[1]{\orgdiv{Department of Mathematics and Computer Science}, \orgname{TU Eindhoven}, \orgaddress{\street{PO Box 513}, \city{Eindhoven}, \postcode{5600 MB}, \country{The Netherlands}}}

\affil[2]{\orgdiv{Department of Mathematics and Informatics, Faculty of Sciences}, \orgname{University of Novi Sad}, \orgaddress{\street{Trg D.Obradovi\'ca 4}, \city{Novi Sad}, \postcode{21000}, \country{Serbia}}}


\abstract{We study the use of inverse harmonic Rayleigh quotients with target for the stepsize selection in gradient methods for nonlinear unconstrained optimization problems.
This provides not only an elegant and flexible framework to parametrize and reinterpret existing stepsize schemes, but also gives inspiration for new flexible and tunable families of steplengths. In particular, we analyze and extend the adaptive Barzilai--Borwein method to a new family of stepsizes.
While this family exploits negative values for the target, we also consider positive targets.
We present a convergence analysis for quadratic problems extending results by Dai and Liao (2002), and carry out experiments outlining the potential of the approaches.}



\maketitle

\section{Introduction} 
We study the unconstrained optimization problem
\[
\min_{\bx \in \mathbb{R}^n} f(\bx)
\]
for strictly convex quadratic and general nonlinear continuously differentiable functions $f: \mathbb{R}^n \to \mathbb{R}$.
We consider the popular gradient method
\[
\bx_{k+1} = \bx_k - \beta_k \, \bg_k \ = \ \bx_k - \alpha_k^{-1} \, \bg_k,
\]
where $\bg_k = \nabla f(\bx_k)$ and $\beta_k > 0$ is the steplength.
It is convenient to introduce a separate notation $\alpha_k$ for the inverse of the stepsize $\beta_k$, since both play important roles; $\alpha_k$ corresponds to (harmonic) Rayleigh quotients, which are scalars providing second-order information (on the Hessian).

As usual, write $\bs_{k-1} = \bx_k-\bx_{k-1}$ and $\by_{k-1} = \bg_k-\bg_{k-1}$.
Two popular stepsizes are the Barzilai--Borwein (BB) steplengths \cite{bb1988}
\[
\step{BB1}k = \frac{\bs_{k-1}^T \bs_{k-1}}{\by_{k-1}^T \bs_{k-1}}, \qquad
\step{BB2}k = \frac{\by_{k-1}^T \bs_{k-1}}{\by_{k-1}^T \by_{k-1}}.
\]
We denote their inverses by $\invstep{BB1}k$ and $\invstep{BB2}k$, respectively.
In case of convex quadratic problems
\begin{equation} \label{quad}
\min_{\bx \in \mathbb{R}^n} \ \tfrac12 \, \bx^T\!A\bx - \bb^T\bx,
\end{equation}
where $A$ is $n \times n$ symmetric positive definite (SPD), and $\bb \in \mathbb{R}^n$, the BB steps are the inverses of the Rayleigh quotient and harmonic Rayleigh quotient,
\[
\step{BB1}k = \frac{\bs_{k-1}^T \bs_{k-1}}{\bs_{k-1}^T A \, \bs_{k-1}}, \qquad
\step{BB2}k = \frac{\bs_{k-1}^T A \, \bs_{k-1}}{\bs_{k-1}^T A^2 \, \bs_{k-1}}.
\]
We refer to \cite{daniela2018steplength} for a nice recent review on various steplength options.

In this paper, we will consider a general framework for these and other stepsizes by introducing a harmonic Rayleigh quotient including a {\em target} $\tau$.
We recall the harmonic Rayleigh--Ritz extraction for matrix eigenvalue problems in Section~\ref{sec:harm}.
The general form of this extraction features a target $\tau \in \mathbb{R} \, \cup \, \{ \pm \infty \}$.
This target is analyzed and exploited in Section~\ref{sec:frame}, to develop a new general framework for all possible stepsizes. We will see that the BB stepsizes correspond to $\tau = 0$ or $\tau = \pm \infty$.
This may not only add towards a new understanding and interpretation of known strategies, but also suggests new competitive schemes.
Section~\ref{sec:abb} closer studies the Adaptive Barzilai--Borwein method (ABB) \cite{zhou2006abb}, and provides a new theoretical justification for it. We also showcase the potential of the framework by introducing new families generalizing the ABB method.
As is common (see, e.g., \cite{daniela2018steplength}) we first consider the convex quadratic problem.
Convergence results for this case, extending those of \cite{dai2002r}, are presented in Section~\ref{sec:conv}.
The extension of the harmonic steplength to general nonlinear problems is treated in Section~\ref{sec:nonlin}.
Finally, we carry out numerical experiments and summarize some conclusions in Sections~\ref{sec:exp} and \ref{sec:concl}.

\section{Harmonic extraction and harmonic Rayleigh quotients} \label{sec:harm}
The harmonic Rayleigh--Ritz extraction has been introduced in the context of eigenvalue problems (see, e.g., \cite{morgan1991computing, PPV95}, \cite[Sec.~4.4]{Ste01}, \cite{Hoc05}) to extract promising approximate (interior) eigenpairs from a subspace. Consider the eigenproblem $A\bx = \lambda \bx$ for a given square $A$. Although $A$ does not necessarily need to be symmetric or real for the harmonic extraction method, in our optimization context we are interested in SPD matrices $A$.

Suppose that we wish to extract promising approximate eigenpairs from a low-dimensional search space $\calu$ for which the columns of $U \in \mathbb{R}^{n \times d}$ form an orthogonal basis, where usually $d \ll n$. We are interested in finding approximate eigenpairs $(\theta, \bu) \approx (\lambda, \bx)$, where $\bu$ is of the form $\bu = U\bc \approx \bx$, with $\bc \in \mathbb{R}^d$ of unit 2-norm.
The standard Rayleigh--Ritz extraction imposes the Galerkin condition
\[
AU\bc - \theta \, U\bc \perp \calu.
\]
This leads to $d$ approximate eigenpairs $(\theta_j, U\bc_j)$, for $j = 1, \dots, d$, obtained from the eigenpairs $(\theta_j, \bc_j)$ of $U^T\!AU$.

Denote the eigenvalues of $A$ by $0 < \lambda_1 \le \cdots \le \lambda_n$.
The standard Rayleigh--Ritz extraction enjoys a good reputation for exterior eigenvalues (see, e.g., \cite{Par98}), which means the largest or smallest few eigenvalues, in our case of symmetric $A$. However, for interior eigenvalues near a target $\tau \in (\lambda_1, \lambda_n)$, the harmonic Rayleigh--Ritz extraction tends to produce approximate eigenvectors of better quality. This approach works as follows (see, e.g., \cite[Sec.~4.4]{Ste01} for more details).

Let $\tau$ be not equal to an eigenvalue; in the context of eigenvalue problems, $\tau$ is typically chosen close to the eigenvalues of interest. Eigenvalues near $\tau$ are exterior eigenvalues of $(A-\tau I)^{-1}$, which is a favorable situation to impose a Galerkin condition. Therefore, the idea is to impose such a condition involving this shifted and inverted matrix. To avoid having to work with an inverse of a (potentially large) matrix, a modified Galerkin condition
\[
(A-\tau I)^{-1} \, U \, \wt \bc - (\wt \theta-\tau)^{-1} \, U \, \wt \bc \perp (A-\tau I)^2\,\calu
\]
is considered.
We note that this is equivalent to the Galerkin condition $(A-\tau I)^{-1} \,\bu - (\theta-\tau)^{-1} \, \bu \perp (A-\tau I)\,\calu$ for $\bu \in (A-\tau I)\,\calu$, which considers this extraction from a different viewpoint.

This implies that the quantities of interest are $(\wt \theta_j, \wt \bc_j)$, for $j = 1, \dots, d$, the eigenpairs of the pencil $(U^T(A-\tau I)\,AU, \ U^T(A-\tau I)\,U)$; and the associated vectors $\wt \bu_j = U \wt \bc_j$.
This means that the relation between a harmonic Ritz vector $\wt \bu = U \wt \bc$ and the corresponding harmonic Ritz value is
\begin{equation} \label{hrq0}
\wt \theta = \frac{\wt \bu^T (A-\tau I)^2 \, \wt \bu}{\wt \bu^T (A-\tau I)\, \wt \bu} + \tau
\ = \ \frac{\wt \bu^T (A-\tau I)\,A \,\wt \bu}{\wt \bu^T (A-\tau I)\, \wt \bu}.
\end{equation}
We will exploit this quantity in the next section to introduce a general harmonic framework for the choice of steplengths.

\section{A harmonic framework for stepsize selection} \label{sec:frame}
Inspired by \eqref{hrq0}, we now propose and study the use of harmonic Rayleigh quotients of the form
\begin{equation} \label{hrq}
\alpha_k(\tau_k) = \frac{\bs_{k-1}^T (A-\tau_k I)\, A\,\bs_{k-1}}{\bs_{k-1}^T (A-\tau_k I) \, \bs_{k-1}}
\end{equation}
in the context of gradient methods, where the $\tau_k$ are targets that may be varied throughout the process. In contrast to the use of the target for eigenvalue problems, where the $\tau_k$ are typically selected inside or very close to the interval $[\lambda_1, \lambda_n]$ (as discussed in Section~\ref{sec:harm}), we investigate strategies with $\tau$-values outside this interval, as well as schemes where these targets may sometimes be inside.

The stepsize we consider is given by the inverse harmonic Rayleigh quotient
\begin{equation} \label{hss}
\beta_k(\tau_k) = \frac{\bs_{k-1}^T (A-\tau_k I) \, \bs_{k-1}}{\bs_{k-1}^T (A-\tau_k I)\,A \,\bs_{k-1}}.
\end{equation}
We will refer to these steps as ``TBB steps'': Barzilai--Borwein type of steps using a harmonic Rayleigh quotient with target $\tau_k$.
In the rest of this section we will consider various aspects of gradient methods with TBB steps as in \eqref{hss}.
In particular, we will discuss strategies for picking $\tau_k$ in Section~\ref{sec:target} and \ref{sec:cotan}.

\subsection{Properties of the TBB stepsize} 
\label{sec:beta}
\gf{In this section, we discuss some properties of the TBB stepsize for strictly convex quadratic functions, i.e., when the Hessian matrix $A$ is SPD and thus $\bs^T\!A\bs > 0$ for any $\bs\ne 0$.}
Then $\alphaa_k \le \alphab_k$ and therefore $\betab_k \le \betaa_k$; in fact (see, e.g., \cite{daniela2018steplength})
\begin{equation} \label{aratio}
\alphaa_k \, / \, \alphab_k = \betab_k \, / \, \betaa_k = \cos^2(\bs_{k-1}, A\bs_{k-1}).
\end{equation}
The following proposition summarizes several basic but essential properties of stepsize \eqref{hss}. \gf{To ease the notation, we drop the index $k$ whenever it is clear that we are referring to the same iteration in the gradient method.}

\begin{proposition} \label{prop:beta}
Let $\bs \in \mathbb{R}^n$ be not equal to a multiple of an eigenvector of $A$, SPD. The function $\beta(\tau) = \frac{\bs^T (A-\tau I) \, \bs}{\bs^T (A-\tau I) \, A \, \bs}$ enjoys the following properties.
\begin{enumerate}
\item[(i)] $\beta(\tau)$ is defined for all $\tau \in \mathbb{R}$ with exception of $\alphab = \frac{\bs^T \! A^2 \, \bs}{\bs^T \! A \, \bs}$, and is a strictly monotonically decreasing function on $(-\infty, \invstep{BB2}{})$ and $(\invstep{BB2}{}, \infty)$.
\item[(ii)] Alternative expressions are $\beta(\tau) = \betaa \, \frac{\tau-\alphaa}{\tau-\alphab} = \betab \, \frac{\betaa \tau-1}{\betab \tau-1}$.
\item[(iii)] $\beta(0) = \betab$ and $\ds \lim_{\tau \to \pm \infty} \beta(\tau) = \betaa$.
\item[(iv)] For $-\infty < \tau < 0$, it holds $\betab < \beta(\tau) < \betaa$.
\item[(v)] For $0 < \tau < \lambda_1$, we have $\frac{1-\betaa \lambda_1}{1-\betab \lambda_1} \cdot \betab < \beta(\tau) < \betab$.
\item[(vi)] For $\tau > \lambda_n$, it holds that $\betaa < \beta(\tau) < \frac{\lambda_n-\alphaa}{\lambda_n-\alphab} \cdot \betaa$.
\item[(vii)] $\beta$ is a bijection from $\mathbb{R} \backslash \{ \alphab \}$ to $\mathbb{R} \backslash \{ \betaa \}$, and from $\mathbb{R} \cup \{\pm \infty\}-\{\alphab\}$ to $\mathbb{R}$.
\end{enumerate}
\end{proposition}
\begin{proof}
The derivative of $\beta$ with respect to $\tau$ is given by
\[
\beta'(\tau) = \frac{(\bs^T \! A \, \bs)^2 - (\bs^T \! A^2 \, \bs) \, (\bs^T \bs)}{(\bs^T (A-\tau I) \, A \, \bs)^2}.
\]
The numerator is equal to $\|A\bs\|^2 \, \|\bs\|^2 \cos^2(A\bs, \bs) - \|A\bs\|^2 \, \|\bs\|^2 < 0$, since $\bs$ is assumed to be not equal to an eigenvector; part (i) follows from this.
Item (ii) is obtained by factoring out $\bs^T \bs$ in the numerator, and $\bs^T \! A \bs$ in the denominator.
Part (iii) follows directly from (ii).
Since $\beta$ is defined everywhere and strictly decreasing on the interval $(-\infty, 0)$ we get item (iv).
Part (v) is derived from (ii) by the fact that $\beta$ on the interval $(0, \lambda_1)$ is defined everywhere and strictly decreasing.
The factor $\frac{1-\betaa \lambda_1}{1-\betab \lambda_1}$ is less than one, since $\betaa,\, \betab < \lambda_1^{-1}$ (again by the fact that $\bs$ is not a multiple of an eigenvector) and $\betab < \betaa$.
Item (vi) follows from the fact that $\beta$ is defined everywhere and strictly decreasing on the interval $(\lambda_n, \infty)$.
The factor $\frac{\lambda_n-\alphaa}{\lambda_n-\alphab}$ is greater than one in view of $\alphaa,\, \alphab< \lambda_n$ and $\alphaa < \alphab$.
\end{proof}

It is particularly item (vii) that implies that {\em the harmonic Rayleigh quotient forms a framework or parametrization for all possible steplengths}: together with target $\tau = \pm \infty$, we have a one-to-one relation between targets in $\mathbb{R} \cup \{ \pm \infty \}$ and any real stepsize (positive or negative).
We stress that, because of the pole of $\beta$ in $\tau = \alphab$, the stepsize might be unbounded for $\tau \in (\lambda_1, \lambda_n)$, which evidently is unwanted. 
Note that in the (unlikely) case that $\bs$ is equal to an eigenvector corresponding to eigenvalue $\lambda$, $\beta(\tau)$ is equal to the constant function $\beta(\tau) \equiv \lambda^{-1}$ (with exception of the ``hole'' at $\tau = \lambda$).
Figure~\ref{fig:hrq} gives an impression of the properties in Proposition~\ref{prop:beta} for a typical situation.

\begin{figure}[htb!]
\centering \includegraphics[width=0.9\textwidth]{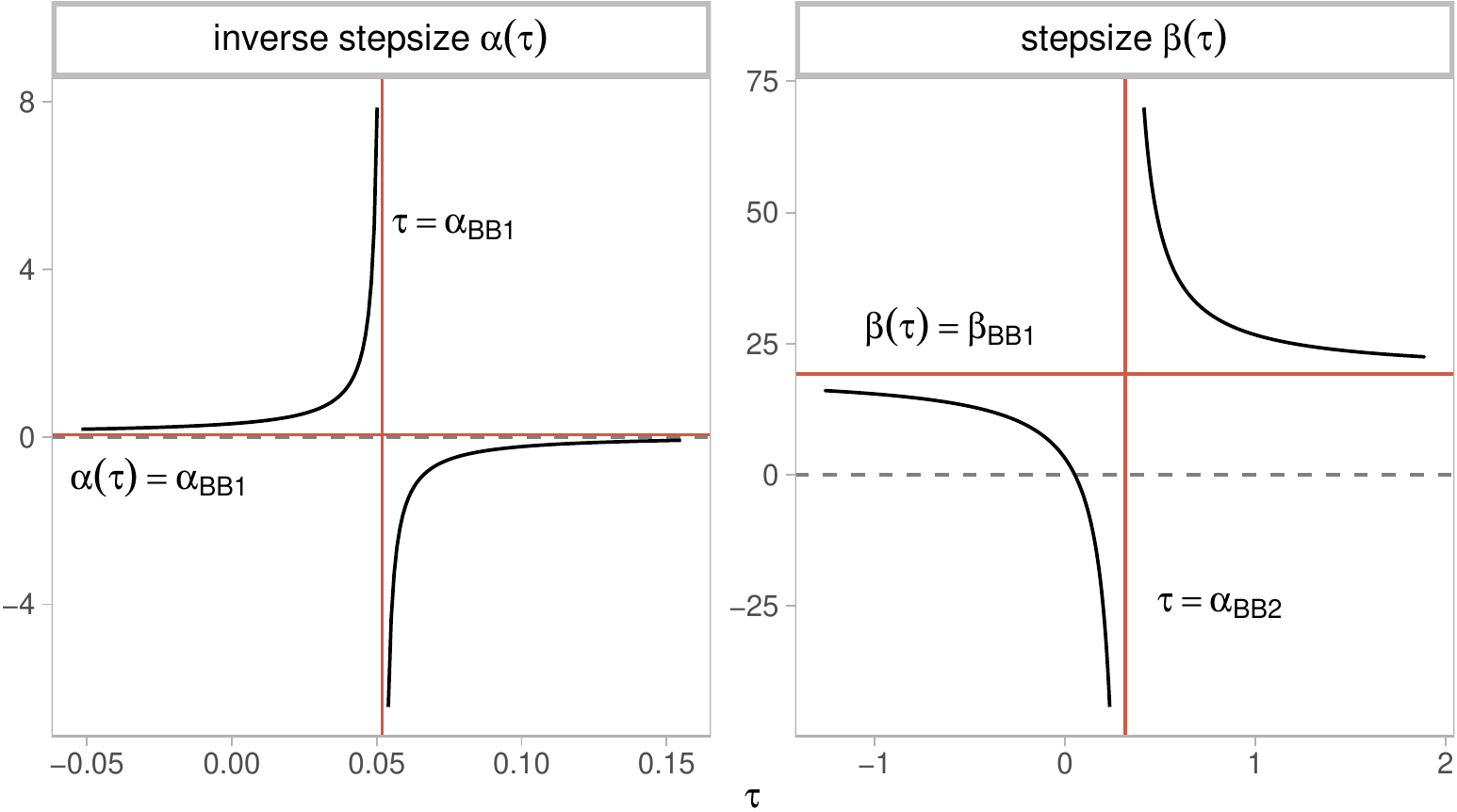}
\caption{Harmonic Rayleigh quotient (left) and its inverse, the stepsize (right), as a function of $\tau$ for the convex quadratic case $A = \text{diag}(\frac1{100}, \frac{1}{99}, \dots, \frac12, 1)$, where $\bs = (1,\dots,1)^T$.}
\label{fig:hrq}
\end{figure}

For completeness, we also list some characteristics of the inverse stepsize $\alpha(\tau)$ (see \eqref{hrq}), the harmonic Rayleigh quotient.

\begin{proposition}
Let $\bs \in \mathbb{R}^n$ be not equal to a multiple of an eigenvector of $A$, SPD.
\begin{enumerate}
\item[(i)] The function $\alpha(\tau) = \frac{\bs^T (A-\tau I) \, A \, \bs}{\bs^T (A-\tau I) \, \bs}$ is defined for all $\tau \backslash \{ \alphaa \}$, and is a strictly monotonically increasing function on the intervals $(-\infty, \alphaa)$ and $(\alphaa, \infty)$.
\item[(ii)] Alternative expression are $\alpha(\tau) = \alphaa \, \frac{\tau-\alphab}{\tau-\alphaa} = \alphab \, \frac{\betab \, \tau-1}{\betaa \, \tau-1}$.
\item[(iii)] $\ds \lim_{\tau \to \pm \infty} \alpha(\tau) = \alphaa$ and $\alpha(0) = \alphab$.
\end{enumerate}
\end{proposition}
\begin{proof}
The derivative of $\alpha$ with respect to $\tau$ satisfies
\[
\alpha'(\tau) = \frac{(\bs^T \! A^2 \, \bs) \, (\bs^T \bs) - (\bs^T \! A \, \bs)^2}{(\bs^T (A-\tau I) \, \bs)^2}.
\]
The result now follows from a reasoning similar to Proposition~\ref{prop:beta}.
Part (ii) can be derived by factoring out a factor of $\bs^T\bs$, $\bs^T\!A\bs$, or $\bs^T\!A^2\bs$ from the numerator or denominator. Item (iii) is straightforward.
\end{proof}


\subsection{Sensitivity of the stepsize with respect to the target} 
\label{sec:sens}
We now study the sensitivity of the steplength $\beta(\tau)$ as function of $\tau$, in particular around $\tau = 0$ and $\tau = -\infty$, which correspond to $\betaa$ and $\betab$, respectively.
We first consider the situation of small $\tau$; recall that $\beta(0) = \betab$.

\begin{proposition}
For $\tau \to 0$, we have up to higher-order terms in $\tau$
\begin{align*}
\frac{\beta(\tau)-\betab}{\betab} & = -\tau \, (\betaa - \betab).
\end{align*}
\end{proposition}
\begin{proof}
For $\tau \to 0$ it holds that (cf.~Proposition~\ref{prop:beta}(ii))
\begin{align*}
\beta(\tau) = \betab \cdot \frac{1-\tau \, \betaa}{1-\tau \, \betab} = \betab \cdot (1-\tau \, (\betaa - \betab)) + \calo(\tau^2).
\end{align*}
\end{proof}

In agreement with Figure~\ref{fig:hrq} and Proposition~\ref{prop:beta}, an appreciable interpretation of this result is that for small negative $\tau$, the stepsize $\beta(\tau)$ increases from $\betab$ (for $\tau = 0$) towards the larger stepsize $\betaa$ (corresponding to $\tau = -\infty$). Moreover, the rate of change for $\tau \to 0$ is asymptotically proportional to the difference between $\betaa$ and $\betab$.
As a side note, from \eqref{aratio} we have that
$\betaa - \betab = \betab \tan^2(\bs,\by)$.

Next, let us investigate the asymptotic situation $\tau \to \pm \infty$.
To this end, we exploit the transformed variable $\zeta = \tau^{-1}$ and consider the expression $\wh \beta(\zeta) := \beta(\zeta^{-1}) = \beta(\tau) = \frac{\zeta \ \bs^T \! A \, \bs - \bs^T\bs}{\zeta \ \bs^T \! A^2 \, \bs - \bs^T \! A \, \bs}$ for $\zeta \to 0$.

\smallskip
\begin{proposition}
For $\tau \to \pm \infty$, we have up to higher-order terms in $\tau^{-1}$
\begin{align*}
\frac{\beta(\tau)-\betaa}{\betaa} & = -\tau^{-1} \, (\alphaa - \alphab).
\end{align*}
\end{proposition}
\begin{proof}
For $\zeta \to 0$ it holds that (cf.~Proposition~\ref{prop:beta}(ii))
\begin{align*}
\wh \beta(\zeta) = \betaa \cdot \frac{1-\zeta \, \alphaa}{1-\zeta \, \alphab} = \betaa \cdot (1-\zeta \, (\alphaa - \alphab)) + \calo(\zeta^2).
\end{align*}
\end{proof}

\noindent
Again, this result has a nice meaning: for small negative $\tau^{-1}$ (i.e., large negative $\tau$), the stepsize $\beta(\tau^{-1})$ decreases from $\betaa$ (for $\tau^{-1} = 0$) towards the smaller stepsize $\betab$ (associated with $\tau=0$).
Moreover, the more $\betaa$ differs from $\betab$, the faster $\beta_k(\tau^{-1})$ decreases as function of $\tau^{-1}$.
For small positive $\tau^{-1}$ (which means large positive $\tau$), the steplength increases, and thus gets larger than $\betaa$; cf.~Figure~\ref{fig:hrq}.
In Sections~\ref{sec:target}, \ref{sec:cotan} and \ref{sec:exp} we will discuss and experiment with strategies involving both negative and positive values of $\tau$.

\subsection{Pseudocode for gradient method with TBB steps} 
In Algorithm~\ref{algo:tbb} we give a pseudocode for a gradient method based on TBB steps.
We exploit a relative stopping criterion in line~4, which may be replaced by any other reasonable stopping rule.

\begin{algorithm}
\caption{A TBB method for strictly convex quadratic functions}
\label{algo:tbb}
{\bf Input}: function $f(\bx) = \frac12 \, \bx^T\!A\bx-\bb^T\bx$ with $A$ SPD, initial guess $\bx_0$, initial stepsize $\beta_0 > 0$, tolerance {\sf tol} \\
{\bf Output}: approximation to minimizer $\argmin_{\bx} f(\bx)$ \\
\begin{tabular}{ll}
{\footnotesize 1}: & Set $\bg_0 = A\bx_0-\bb$ \\
& {\bf for} $k = 0, 1, \dots$ \\
{\footnotesize 2}: & \phantom{M} Set $\bs_k = -\beta_k \, \bg_k$ and update $\bx_{k+1} = \bx_k + \bs_k$ \\
{\footnotesize 3}: & \phantom{M} Compute the gradient $\bg_{k+1} = A\bx_{k+1}-\bb$ \\
{\footnotesize 4}: & \phantom{M} {\bf if} \ $\|\bg_{k+1}\| \le {\sf tol} \cdot \|\bg_0\|$, \ {\bf return}, \ {\bf end} \\
{\footnotesize 5}: & \phantom{M} Choose $\tau_{k+1}$, compute TBB step $\beta_{k+1}(\tau_{k+1})$ according to \eqref{hss} \\
\end{tabular}
\end{algorithm}

Clearly, the choice of targets $\tau_k$ in Line~5 is a crucial aspect of the method. We discuss some options for this particularly in Section~\ref{sec:target} and \ref{sec:cotan}.
In Section~\ref{sec:nonlin} we also consider practically important details such as the choice of $\beta_0$.

\subsection{Secant conditions} 
\label{sec:secant}
In this subsection we consider an equivalent formulation of the TBB stepsize
\begin{equation} \label{hssg}
\beta_k(\tau) = \frac{\bs_{k-1}^T \, (\by_{k-1}-\tau_k \ \bs_{k-1})}{\by_{k-1}^T \, (\by_{k-1}-\tau_k \ \bs_{k-1})},
\end{equation}
where $\by_{k-1} = A\bs_{k-1}$; this will be useful in Section~\ref{sec:nonlin} for generic problems where the Hessian changes over the iterations.

We recall from \cite{bb1988} that a justification of the BB steps is the fact that they approximate the Hessian matrix by a scalar multiple of the identity, as follows. It is reasonable that an approximation $B_k$ to the Hessian approximately satisfies the secant equation $\by_{k-1} = B_k \bs_{k-1}$.
The BB steps solve the secant equation in a least-squares sense \cite{bb1988}:
\begin{equation} \label{bbsec}
\invstep{BB1}k = \argmin_\alpha \|\by_{k-1} - \alpha \, \bs_{k-1}\|, \qquad \invstep{BB2}k = \argmin_\alpha \|\bs_{k-1} - \alpha^{-1} \, \by_{k-1}\|,
\end{equation}
which results in an approximation of the form $B_k = \alpha I$ to the Hessian, where $\alpha$ is $\invstep{BB1}k$ or $\invstep{BB2}k$, respectively.

As the TBB step \eqref{hss} involves the shifted matrix $A - \tau I$ (where $\tau$ may vary over the iterations), this suggests us to consider a shifted secant equation
\begin{equation} \label{bbsec-shift}
\by_{k-1} - \tau \, \bs_{k-1} = (B_k - \tau I)\,\bs_{k-1}.
\end{equation}
By replacing $\by_{k-1}$ by $\by_k - \tau \, \bs_{k-1}$ and $\alpha$ by $\alpha-\tau$ in the second secant condition in \eqref{bbsec}, we obtain that the TBB step satisfies a modified secant condition, which is equivalent to the second equation of \eqref{bbsec} for $\tau = 0$, but not equivalent to the first or second one for any other target value.

\begin{proposition} \label{prop:bbsec-shift}
Let $\tau \ne \invstep{BB1}k = \frac{\by_{k-1}^T\bs_{k-1}}{\bs_{k-1}^T\bs_{k-1}}$. Then the inverse TBB step $\alpha_k = \beta_k^{-1}$ satisfies
\begin{equation}
\label{secant-cond-tbb}
\alpha_k(\tau) = \frac{\by_{k-1}^T \, (\by_{k-1} - \tau \, \bs_{k-1})}{\bs_{k-1}^T \, (\by_{k-1} - \tau \, \bs_{k-1})} = \argmin_\alpha \|\bs_{k-1} - (\alpha-\tau)^{-1}(\by_{k-1} - \tau \, \bs_{k-1})\|.
\end{equation}
\end{proposition}
\begin{proof}
The result follows by setting to zero the derivative of the square of the objective function in \eqref{secant-cond-tbb} with respect to $\alpha$, which gives
\[
(\by_{k-1} - \tau \, \bs_{k-1})^T(\bs_{k-1} - (\alpha - \tau)^{-1}(\by_{k-1} - \tau \, \bs_{k-1})) = 0.
\]
\end{proof}

In addition to this interpretation as modified secant condition, when the target is located outside $[\lambda_1, \lambda_n]$, we can also think of the TBB step as the scalar least squares solution to the following problem involving a certain weighted norm, as follows. 
Define the standard weighted norm associated with a given SPD matrix $W$ by $\|\bx\|_W^2 := \bx^TW\bx$.

\begin{proposition}
\label{lemma:min-norm}
The least squares solution to the weighted secant equation satisfies
\begin{equation} \label{secant-equ}
\argmin_\alpha \| \by_{k-1} - \alpha \, \bs_{k-1}\|_W = \frac{\by_{k-1}^T \, W \, \bs_{k-1}}{\bs_{k-1}^T \, W \, \bs_{k-1}}.
\end{equation}
\end{proposition}
\begin{proof}
The result follows by setting $-\by_{k-1}^TW\bs_{k-1} + \alpha \, \bs_{k-1}^TW\bs_{k-1}$, the derivative of $\frac12 \, \| \by_{k-1} - \alpha \, \bs_{k-1}\|_W^2$ with respect to $\alpha$, to zero.
\end{proof}

The BB1 and BB2 steps can be obtained from this proposition by taking $W=I$ and $W=A$, respectively; cf.~\eqref{bbsec}. The TBB step can be derived by choosing $W = A-\tau I$ for $\tau < \lambda_1$ or $W = \tau I-A$ for $\tau > \lambda_n$, which gives an SPD weight matrix in both cases. 

In conclusion, the BB1, BB2, and TBB steps approximate the Hessian by a positive scalar multiple of the identity of the form $B_k = \alpha_k I \approx A$. The scalars satisfy one or both (weighted) secant conditions.

\subsection{Regularization} 
Another viewpoint on harmonic steps \eqref{hss} with a target $\tau$ outside $[\lambda_1, \lambda_n]$ is as {\em regularization of the Hessian}.
First consider taking a shift $\tau < 0$. As we replace $A$ by $A-\tau I$ for $\tau < 0$ this yields a ``more positive definite'' shifted Hessian.
As one indicator, the condition number
\begin{equation} \label{condnr}
\kappa(A) = \lambda_n \, / \, \lambda_1
\end{equation}
of $A$ is modified to $\frac{\lambda_n-\tau}{\lambda_1-\tau}$ by this shift.
Consider the function $\varphi: (-\infty, 0] \to [1, \infty)$ given by $\varphi(t) = \frac{\lambda_n-t}{\lambda_1-t}$. Since this function is strictly monotonically decreasing on the domain $(-\infty, 0]$, we conclude that $\kappa(A-\tau I) < \kappa(A)$.
More precisely, we have the following first-order estimate.

\begin{proposition}
For $\tau \to 0$ we have
\[
\kappa(A-\tau I) = \kappa(A) + \frac{\lambda_n-\lambda_1}{\lambda_1^2} \cdot \tau + \calo(\tau^2).
\]
\end{proposition}
\begin{proof}
Straightforward using the linear approximation $\varphi(t) \approx \varphi(0) + \varphi'(0)\, t$.
\end{proof}

In fact, for $\tau < 0$, the shifted condition number $\kappa(A-\tau I)$ may be considerably smaller than $\kappa(A)$, especially if the Hessian is nearly singular.
In conclusion, also in view of \eqref{bbsec-shift} and Proposition~\ref{prop:bbsec-shift}, {\em harmonic stepsizes with $\tau < 0$ may be viewed as satisfying a secant condition on a regularized Hessian}. Note that $\ds \lim_{\tau \to -\infty} \kappa(A-\tau I) = 1$.

Moreover, for $\tau > \lambda_n$, we have a similar situation. It is not difficult to show that $\kappa(A-\tau I) = \kappa(A)$ when $\tau = \lambda_n+\lambda_1$ (which is usually close to $\lambda_n$). For $\tau > \lambda_n+\lambda_1$, the condition number of the shifted matrix $A-\tau I$ decreases monotonically, with the analogous property $\ds \lim_{\tau \to \infty} \kappa(A-\tau I) = 1$.


\subsection{Connections with other stepsizes} 
\label{sec:other}
We would like to point out that quotients of the form
\begin{equation} \label{prq}
\frac{\bs^T \, p(A)\, A\,\bs}{\bs^T \, p(A) \,\bs},
\end{equation}
for certain polynomials $p$, have also been considered in different contexts in \cite[(2.6)]{friedlander1998gradient}, \cite{Hoc05}, and \cite{huang2022acceleration}.
The harmonic Rayleigh quotient \eqref{hrq} is a special case of \eqref{prq}, but a very practical instance for several reasons.
First, it gives a clear connection with the harmonic Rayleigh--Ritz extraction for eigenvalue problems as seen in Section~\ref{sec:harm}.
Second, as we have seen in Section~\ref{sec:beta}, by taking first-order polynomials $p$ in $\tau$, we have a one-to-one correspondence between the target $\tau$ and the stepsize $\beta$.

The introduction of an adjustable parameter in the stepsize has been first proposed in \cite{dai2019family}, where the authors present a convex combination of BB1 and BB2 steps,
\[
\step{CON}{k}(\zeta_k) = \zeta_k \, \step{BB1}{k} + (1-\zeta_k) \, \step{BB2}{k},
\]
where $\zeta_k \in [0,1]$. Note that $\step{CON}{k}(0) = \step{BB2}{k}$ and $\step{CON}{k}(1) = \step{BB1}{k}$. Its inverse minimizes a linear combination of secant conditions, i.e.,
\[
\argmin_{\alpha} \|\zeta \, (\alpha \, \bs_{k-1}-\by_{k-1}) + (1-\zeta) \, (\bs_{k-1}-\alpha^{-1} \, \by_{k-1})\|.
\]
In \cite{dai2019family}, several strategies are considered to choose $\zeta_k$: fixed, randomly from the uniform distribution over $[0,1]$, or 
imitating the behavior of the cyclic gradient methods (cf.~\cite[pp.~56--57]{dai2019family} and references therein).
In the next section, we show a link between this convex combination steplength and the TBB step. This not only suggests relevant strategies to select $\zeta_k$, or rather $\tau_k$ for our TBB methods, but also gives a far wider range of options.

\subsection{Strategies to select targets} 
\label{sec:target}
There is a one-to-one correspondence between the parameter $\zeta_k$ in Section~\ref{sec:other} and the target $\tau_k$ in the TBB stepsize: with the choice 
\begin{equation} \label{eq:targetconv}
\tau_k = -\frac{\zeta_k}{1-\zeta_k} \, \invstep{BB2}{k},
\end{equation}
the corresponding TBB step coincides with the stepsize in \cite{dai2019family}: $\step{}{k}(\tau_k) = \step{CON}{k}(\zeta_k)$. In addition, since $\zeta_k \in [0,1]$, the corresponding target values lie in $\tau_k \in [-\infty, 0]$. From Proposition~\ref{prop:beta} we conclude that $\step{BB2}{k} \le \step{CON}{k}(\zeta_k) \le \step{BB1}{k}$. 

Given the relation between $\zeta_k$ and $\tau_k$, all strategies mentioned in \cite{dai2019family} correspond to negative targets $\tau_k$ for the TBB steplengths \eqref{hss}. In the next section, we analyze new schemes for the choice of negative targets; here we focus on positive targets $\tau_k$.
As this yields steplengths $\beta_k(\tau_k) > \step{BB1}k$, this is not equivalent to any of the stepsizes determined by $\zeta_k$ in \cite{dai2019family}. Inspired by the expression in \eqref{eq:targetconv}, where the target is a negative factor times the inverse BB2 stepsize $\alphab_k$, we consider positive targets of the form
\begin{equation}
\label{eq:t1}
\tau_k = \toneconst \, \invstep{BB2}{k}, \quad \toneconst>1.
\end{equation}
This gives us an {\em affine} (rather than convex) combination of BB1 and BB2:
\begin{equation*}
 \step{}{k}(\tau_k)= \tfrac{\toneconst}{\toneconst-1} \, \step{BB1}{k} -\tfrac{1}{\toneconst-1} \, \step{BB2}{k}.
\end{equation*}
This new stepsize is located in the right branch of the hyperbola (right plot of Figure \ref{fig:hrq}). We will make use of the following bounds for the inverse stepsize:
\[
\invstep{}{k}(\tau_k) = \tfrac{\toneconst-1}{\toneconst-\cos^2(\bs_{k-1}, \, \by_{k-1})} \, \invstep{BB1}{k} \ \in \ [\tfrac{\toneconst-1}{\toneconst} \, \invstep{BB1}{k}, \, \invstep{BB1}{k}] \ \subseteq \ [\tfrac{\toneconst-1}{\toneconst}\,\lambda_1, \ \lambda_n].
\]
As in \eqref{eq:targetconv}, we may let $\toneconst$ vary through the iterations.
In the numerical experiments, we will consider the strategy $\tau_1 = 0$ and
\begin{equation} \label{eq:t2}
\tau_k = k\,\invstep{BB2}{k}, \qquad \text{for} \ k=2,3,\dots.
\end{equation}
Since $\tau_k \ge k\,\lambda_1$ for all $k$, the sequence $\{\tau_k\}$ converges to infinity; therefore, the corresponding inverse stepsizes $\beta_k(\tau_k)$ behave asymptotically as $\invstep{BB1}{k}$. In the long run, there exists a $\toneconst > 1$ such that $\invstep{}{k}(\tau_k) \in [\frac{\toneconst-1}{\toneconst}\lambda_1, \, \lambda_n]$ for all $k > \toneconst -1$.

The inverse stepsizes obtained from \eqref{eq:t1} and \eqref{eq:t2} are bounded from below by a multiple of $\invstep{BB1}{k}$, which will play a key role in the global convergence of the resulting gradient method in Section~\ref{sec:conv}.

\section{An analysis and extension of the ABB scheme} \label{sec:abb}
We now present an analysis and ``continuous extension'' of the ABB method \cite{zhou2006abb}. We propose a new harmonic family of stepsizes, using various options for the target $\tau_k$ throughout the process. These strategies aim to retain the advantages of the ABB method, while being more flexible and tunable than the original approach. 

The motivations for this adaptation are the following. First, it is a popular method. Second, the method contains a threshold parameter $\eta$, the choice of which may be seen as a bit arbitrary.
Third, it is a good showcase of the possibilities that the harmonic framework offers.

We first briefly recall the ABB approach. The stepsize is selected as
\[
\step{ABB}k = \left\{
\begin{array}{ll}
\step{BB2}k, & \ \text{if} \ \step{BB2}k < \eta \, \step{BB1}k, \\[2mm]
\step{BB1}k, & \ \text{otherwise.}
\end{array}
\right.
\]
Here, $\eta$ is a user-selected parameter with a common choice $\eta = 0.8$; see, e.g., \cite{daniela2018steplength}.
\gf{As in the previous section, we study the strictly convex quadratic case, so that $\step{BB2}k \le \step{BB1}k$.} The ABB scheme adaptively picks BB1 and BB2 steps based on the value of $\step{BB2}k \, / \, \step{BB1}k = \cos^2(\bs_{k-1}, \by_{k-1})$. 

The idea of the ABB stepsize is to take a larger step when $\cos^2(\bs_{k-1}, \by_{k-1}) \approx 1$, and a smaller step when this is not the case. If $\cos^2(\bs_{k-1}, \by_{k-1})$ is close to $1$, this means that $\bg_{k-1}$ (or, equivalently, $\bs_{k-1}$) is close to an eigenvector of $A$ corresponding to an eigenvalue $\lambda > 0$. Thus, as we will see in Section~\ref{sec:conv}, by the step $\step{BB1}{k}$ we are particularly reducing the gradient component corresponding to $\lambda$ significantly. When we are far from an eigenvalue, we prefer to take shorter steps, such as $\step{BB2}{k}$, since we aim to reduce several gradient components; this fosters the gradient method to take a new longer step in the next iterations.
There exist several variants of the ABB method; the interested reader may refer to \cite{frassoldati2008new, zhou2006abb, bonettini2009scaled} for such ideas. 

\subsection{A theoretical foundation for the ABB method} 
As discussed, a key statement for the ABB method is: ``if $\cos(\bs_{k-1},\by_{k-1}) \approx 1$, then $\bs_{k-1}$ is close to an eigenvector'' (see, e.g., \cite[pp.~179--180]{daniela2018steplength}).
The following new result quantifies this statement for the quadratic case.
We need the assumption that $\lambda$ is a simple eigenvalue, since otherwise an eigenvector is not uniquely defined.
We consider the situation $\cos^2(\bs, A\bs) \approx 1$, so that $\sin(\bs, A\bs)$ is small.

\begin{proposition} 
\label{prop:abb-cos}
Let $(\lambda, \bx)$ be an eigenpair of $A$, where $\lambda$ is a simple eigenvalue. Let $\bs \approx \bx$ be an approximate eigenvector.
Then, up to higher-order terms in $\angle(\bs,\bx)$,
\begin{equation*}
\frac{\lambda}{\max_{\lambda_i \ne \lambda} \vert\lambda_i-\lambda\vert} \cdot \sin(\bs, A\bs) \lesssim \tan(\bs,\bx) \lesssim \frac{\lambda}{\min_{\lambda_i \ne \lambda} \vert\lambda_i-\lambda\vert} \cdot \sin(\bs, A\bs).
\end{equation*}
\end{proposition}
\begin{proof}
Without loss of generality we may assume that $A = \text{diag}(\lambda, \Lambda)$, where $\Lambda$ is an $(n-1) \times (n-1)$ diagonal matrix containing all eigenvalues different from $\lambda$, and that $\bs$ is of the form $[1, \ \bz]^T$,
a perturbation of $\bx$, the first canonical basis vector. This means $\tan(\bs,\bx) = \|\bz\|$; our goal is to connect this quantity to $\angle(\bs, A\bs)$ via $\sin(\bs, A\bs)$.
We have $As = [\lambda, \ \Lambda\,\bz]^T$,
and
\[
\cos^2(\bs,A\bs) = \frac{(\bs^T\!A\bs)^2}{\|\bs\|^2 \, \|A\bs\|^2} = \frac{(1 + \lambda^{-1} \, \bz^T\Lambda\,\bz)^2}{(1+\|\bz\|^2) \, (1 + \lambda^{-2} \, \|\Lambda\,\bz\|^2)}.
\]
We now twice use the Taylor expansion $(1-t)^{-1} = 1+t + \calo(t^2)$ for small $t$, so that
\[
\begin{array}{ll}
(1+\|\bz\|^2)^{-1} & = 1 - \|\bz\|^2 + \calo(\|\bz\|^4), \\[2mm]
(1+\lambda^{-2} \, \|\Lambda\,\bz\|^2)^{-1} & = 1 - \lambda^{-2} \, \|\Lambda\,\bz\|^2 + \calo(\|\bz\|^4).
\end{array}
\]
This yields that
\[
\begin{array}{ll}
\sin^2(\bs, A\bs) & 
= \gf{1 - (1+2\lambda^{-1} \, \bz^T\Lambda\,\bz) \, (1-\|\bz\|^2) \, (1-\lambda^{-2} \, \|\Lambda \, \bz\|^2) + \calo(\|\bz\|^4)} \\[1mm]
& = \|\bz\|^2 + \lambda^{-2} \, \|\Lambda\,\bz\|^2 - 2\lambda^{-1} \, \bz^T\Lambda\,\bz + \calo(\|\bz\|^4) \\[1mm]
& = \|\bz-\lambda^{-1}\, \Lambda\,\bz\|^2 + \calo(\|\bz\|^4).
\end{array}
\]
Multiplication by $\lambda > 0$ and omitting $\calo(\|\bz\|^4)$-terms gives $\|\Lambda\,\bz-\lambda \bz\| = \lambda \, \sin(\bs, A\bs)$.
The result now follows from
\[
\min_{\lambda_i \ne \lambda} \vert\lambda_i-\lambda\vert \ \|\bz\| \le \|\Lambda\,\bz-\lambda \bz\| \le \max_{\lambda_i \ne \lambda} \vert\lambda_i-\lambda\vert \ \|\bz\|.
\]
\end{proof}

Next, we investigate the sensitivity of the BB steps for quadratic problems where the direction is close to an eigenvector.

\begin{proposition} \label{prop:sens-bb}
Let $\bs$ be an approximation of an eigenvector $\bx$ corresponding to a simple eigenvalue $\lambda$.
Up to higher-order terms in $\angle(\bs,\bx)$, we have 
\begin{align*}
\frac{\vert\step{BB1}{k} - \lambda^{-1}\vert}{\lambda^{-1}} & \lesssim \frac{\max_i \vert\lambda_i-\lambda\vert}{\lambda} \cdot \tan^2(\bs,\bx), \\[2mm]
\frac{\vert\step{BB2}{k} - \lambda^{-1}\vert}{\lambda^{-1}} & \lesssim 
\frac{\max_i \vert\lambda_i\,(\lambda_i-\lambda)\vert}{\lambda^2} \cdot \tan^2(\bs,\bx).
\end{align*}
\end{proposition}
\begin{proof}
With the same notation as in the proof of Proposition~\ref{prop:abb-cos}, we have for the BB1 step
\[
\frac{\bs^T\bs}{\bs^T\!A\,\bs} = \lambda^{-1} \, \frac{1+\bz^T\bz}{1 + \lambda^{-1} \, \bz^T\Lambda\,\bz}
= \lambda^{-1} \, (1 + \bz^T (1-\lambda^{-1} \Lambda)\,\bz) + \calo(\|\bz\|^4),
\]
and for the BB2 step
\[
\frac{\bs^T\!A\,\bs}{\bs^T\!A^2\,\bs} = \frac{\lambda+\bz^T\Lambda\,\bz}{\lambda^2 + \bz^T\Lambda^2\,\bz}
= \lambda^{-1} \, (1 + \bz^T \lambda^{-1}\Lambda \, (I-\lambda^{-1} \Lambda)\,\bz) + \calo(\|\bz\|^4).
\]
\end{proof}

Comparing these two upper bounds, we conclude that the one for the BB2 steps may be larger by a factor
$\lambda_n / \lambda$, which is close to $\kappa(A)$ for small $\lambda$ and may therefore be very large.
As a result, one interpretation of this proposition is that, indeed, it may be a good idea to take BB1 steps rather than BB2 steps if $\bs$ is close to an eigenvector, since BB2 steps are more sensitive with respect to perturbations in that direction, in view of the extra factor in the upper bound.
This seems especially relevant for small $\lambda$, corresponding to large stepsizes.
Therefore, this result forms a clear mathematical motivation for the ABB scheme.

\subsection{Sensitivity of $\beta(\tau)$ with respect to $\bs$} 
The following result is an extension of Proposition~\ref{prop:sens-bb} to the harmonic step with target \eqref{hss}.
We will see that it reduces to Proposition~\ref{prop:sens-bb} in the case of $\tau = 0$ or $\tau \to \pm \infty$.

\begin{proposition}
Let $\bs$ be an approximation of an eigenvector $\bx$ corresponding to a simple eigenvalue $\lambda$.
Up to higher-order terms in $\angle(\bs,\bx)$, we have 
\[
\frac{\vert\beta(\tau) - \lambda^{-1}\vert}{\lambda^{-1}} \lesssim \max_i \frac{\vert\lambda_i-\tau\vert \cdot \vert\lambda_i-\lambda\vert}{\vert\lambda-\tau\vert \cdot \lambda} \cdot \tan^2(\bs,\bx).
\]
\end{proposition}
\begin{proof}
With the notation as in Proposition~\ref{prop:sens-bb}, and using similar techniques we get
\begin{align*}
\beta(\tau) & =
\frac{\lambda-\tau+\bz^T(\Lambda-\tau I)\,\bz}{\lambda\,(\lambda-\tau) + \bz^T\Lambda \, (\Lambda-\tau I)\,\bz} \\[1mm]
& = \lambda^{-1} \, (1 + (\lambda-\tau)^{-1} \, \bz^T \, (\Lambda-\tau I) \, (I-\lambda^{-1} \Lambda)\,\bz) + \calo(\|\bz\|^4).
\end{align*}
\end{proof}

Interestingly, the factor $\vert\lambda_i-\tau\vert \, / \, \vert\lambda-\tau\vert$ converges to $1$ for $\tau \to \pm \infty$, which reduces to the first inequality in Proposition~\ref{prop:sens-bb}; this corresponds to the BB1 step, the inverse Rayleigh quotient.
When $\tau \to 0$, this factor converges to that in the second inequality in Proposition~\ref{prop:sens-bb}; this corresponds to the BB2 step, the inverse harmonic Rayleigh quotient for zero target.

\subsection{A new family of stepsizes} 
\label{sec:cotan}
The ABB strategy may be viewed as ``discrete'', in the sense that just two types of stepsizes are possible: the BB1 or the BB2 step.
We will now propose a new ``continuous'' variant of ABB parameterized by choosing appropriate $\tau_k$.
We design this strategy to have a similar behavior as ABB: when $\cos^2(\bs_{k-1}, \by_{k-1}) \approx 1$, the steps are close to the BB1 step, while the steps should be close to the BB2 step when $\cos^2(\bs_{k-1}, \, \by_{k-1}) \approx 0$. Therefore, we are interested in a function of $\cos(\bs_{k-1}, \, \by_{k-1})$ such that when $\tau_k \to -\infty$ we recover BB1, while $\tau_k = 0$ yields the BB2 step. One choice to attain this is to use a cotangent function:
\[
\tau_k = -\cot(\bs_{k-1}, \, \by_{k-1}).
\]
Indeed, this choice has the two types of desired asymptotic behavior.
To further tune the speed by which we approach the two BB steps when $\cos(\bs_{k-1}, \by_{k-1})$ approaches 0 or 1, we will also introduce two extra parameters, and consider
\begin{equation} \label{cot}
\tau_k = -\frac{\cos^q(\bs_{k-1}, \, \by_{k-1})}{\sin^r(\bs_{k-1}, \, \by_{k-1})},
\end{equation}
for $q$, $r > 0$.
For example, if we want our gradient method to have shorter steps more often than long ones, we may keep $r = 1$ but select a higher value of $q$, e.g., $q = 2$. This mimics the effect of setting $\eta$ relatively close to $1$, as it is done in \cite{daniela2018steplength}.
The following result ensures that this ``cotangent step'' \eqref{cot} indeed may be regarded as a continuous extension of the ABB step, having similar properties for $\angle(\bs, \by)$ close to 0 or $\pi/2$.

\begin{proposition}
For $q, r > 0$, consider $\beta(\tau)$ as in \eqref{hss}, where $\tau$ is defined by \eqref{cot}. We have that $\beta(\tau) \to \betaa$ when $\angle(\bs, \by) \to 0$ and $\beta(\tau) \to \betab$ when $\angle(\bs, \by) \to \pi/2$.
Moreover, $\beta(\tau)$ is a decreasing function of $\angle(\bs, \by)$.
\end{proposition}
\begin{proof}
Since $\sin$ is strictly increasing and $\cos$ is strictly decreasing on $(0,\frac{\pi}2)$, the function $\tau$ defined by \eqref{cot} is a strictly increasing function of $\angle(\bs, \by)$, ranging from $-\infty$ for $\angle(\bs, \by) \to 0$ to $0$ for $\angle(\bs, \by) = \frac{\pi}2$.
Therefore, in combination with Section~\ref{sec:beta}, we conclude that $\beta$ decreases from $\betaa$ to $\betab$.
\end{proof}

We point out that in the quadratic case the targets \eqref{cot} are negative, which implies that the corresponding stepsize has the same bounds as $\step{CON}{k}$ in Section~\ref{sec:target}: $\step{BB1}{k} \le \step{}{k}(\tau_k) \le \step{BB2}{k}$. We will test various choices for $q, r > 0$ in the experiments in Section~\ref{sec:exp}.

\section{Convergence analysis} \label{sec:conv}

In this section, we extend a few results on BB steps \gf{for strictly convex quadratics} to the TBB step with the choices for the target $\tau_k$ described in Sections \ref{sec:target} and \ref{sec:cotan}. Global convergence of the gradient method with BB steps has been proven by Raydan \cite{raydan1993barzilai} for strictly convex quadratic functions. Dai and Liao \cite{dai2002r} and Dai \cite{dai2003alternate} show R-linear convergence of the method for a class of BB stepsizes. We follow \cite[Thm.~4.1]{dai2003alternate} and \cite[Thm.~2.5]{dai2002r} to extend the results to the TBB steps.

Before stating the assumptions that the TBB stepsize must satisfy, we introduce some expressions that will be useful for the following results. \gf{Since the Hessian is fixed through the iterations, it is sensible to} decompose the gradient along an orthonormal basis of eigenvectors of $A$. Let $\bv_1, \dots, \bv_n$ be the (orthonormal) eigenvectors associated with the eigenvalues $\lambda_1, \dots, \lambda_n$. The gradient can be expressed as linear combination of these eigenvectors (cf., e.g., \cite[(2.2)]{dai2002r})
\[
g_k = \sum_{i=1}^n \gradcomp{i}k \, \bv_i.
\]
Therefore, the TBB step expressed in the eigenvalues of $A$ and the $\gradcomp{i}{k-1}$ is
\begin{equation} \label{tau-lambda}
\invstep{}{k}(\tau) = \frac{\sum_i (\gradcomp{i}{k-1})^2\,\lambda_i\,(\lambda_i - \tau)}{\sum_i (\gradcomp{i}{k-1})^2\,(\lambda_i - \tau)},
\end{equation}
provided the denominator is nonzero. In particular, for $\tau \to \pm\infty$, this stepsize converges to the BB1 stepsize, which is expressed as (cf., e.g., \cite[Eq.~(2.18)]{dai2002r})
\begin{equation} \label{tau-lambda-bb1}
\invstep{BB1}{k} = \frac{\sum_i \, (\gradcomp{i}{k-1})^2\ \lambda_i}{\sum_i \, (\gradcomp{i}{k-1})^2}.
\end{equation}
For strictly convex quadratic functions, we have the following recursive formula for the gradients $g_k$ and, as a consequence, for their coefficients (cf., e.g., \cite[Eq.~(8)]{raydan1993barzilai})
\begin{align}
\label{recursive-err}
\bg_{k+1} &= (I - \step{}{k} A)\,\bg_k, \\
\label{recursive-coefs}
\gradcomp{i}{k+1} &= (1 - \step{}{k} \lambda_i)\, \gradcomp{i}k.
\end{align}
Equations \eqref{recursive-err}--\eqref{recursive-coefs} can also be applied to the error $\be_k = \bx_k - A^{-1}\bb$ and its components $e_i^k$ in the directions of the eigenvectors. Using these equations, the following results can be obtained for the error components as well. The only complication is that higher powers of the eigenvalues appear:
\[
\invstep{}{k}(\tau) = \frac{\sum_i \, (e_i^{k-1})^2\,\lambda_i^3\,(\lambda_i - \tau)}{\sum_i \, (e_i^{k-1})^2\,\lambda_i^2\,(\lambda_i - \tau)},
\]
since $\step{}{k}\bs_k = A\be_k$ for all $k$. This expression extends \cite[Eq.~(12) and p.~325]{raydan1993barzilai}. To prove the convergence of the gradient method for strictly convex quadratic functions, it is sufficient to show that $\|\bg_k\| \to 0$. Since we chose an orthonormal basis, $\|\bg_k\|^2 = \sum_i (\gradcomp{i}{k})^2$ and thus we aim to show $\gradcomp{i}{k}\to 0$ for $i=1,\dots,n$. We remark that working with the $\gradcomp{i}{k}$ is equivalent to assuming that our Hessian matrix $A$ is diagonal.

\subsection{Assumptions for the TBB stepsize} 
We state the assumptions on the TBB stepsize, needed to get R-linear convergence. We adapt \cite[Property~A]{dai2003alternate} to the TBB steplengths proposed in Sections~\ref{sec:target} and \ref{sec:cotan}:

\begin{assumption}
\label{assumption}
The inverse stepsize $\alpha_{k}$ satisfies Assumption~\ref{assumption} if there exist positive constants $\xia \in (\frac12, 1]$, $\xib\ge 1$ and $M_2$ such that, for any $k$,
\begin{itemize}
\item[(i)] $\xia \cdot \lambda_1\le\invstep{}{k} \le \xib \cdot \lambda_n$;
\item[(ii)] for any $1 \le \ell \le n-1$ and $\eps > 0$, if $\sum_{i=1}^\ell (\gradcomp{i}{k-1})^2 \le \eps$ and $(\gradcomp{\ell+1}{k-1})^2\ge M_2 \, \eps$, then $\alpha_{k} \ge \frac{\xia}{\xia+1/2} \, \lambda_{\ell+1}$.
\end{itemize}
\end{assumption}

\cite[Property~A]{dai2003alternate} also includes retards in the BB steps, but, as we are not interested in retards in this paper, we will not include them to ease the notation in what follows. The interested reader is referred to \cite{friedlander1998gradient} for the definition of BB steps with retards, and to \cite{dai2003alternate} for the proof of R-linear convergence under this property. Secondly, we note that \cite[Property~A]{dai2003alternate} requires $\xia = 1$, while we allow a looser lower bound for the inverse stepsize. In other words, admissible stepsizes are larger than the largest eigenvalue of $A^{-1}$. Finally, our upper bound in (i) has a more specific shape than the one set in \cite[Property~A]{dai2003alternate}, which is some $M_1 \ge \lambda_1$. This will enable us to express some bounds as a function of the condition number of $A$. Given all the differences between Assumption~\ref{assumption} and \cite[Property~A]{dai2003alternate}, it is worthwhile to analyze the new situation.

We show that targets \eqref{eq:t1}, \eqref{eq:t2}, and \eqref{cot} all lead to a stepsize that satisfies Assumption~\ref{assumption}. In addition, we show a useful bound for the $(\ell + 1)$st gradient component. Analogous proofs can be found in, e.g, \cite[Corollary 4.2]{dai2003alternate} or \cite[Lemma 2]{raydan1993barzilai}. First notice that, with these choices for the target, for $k$ sufficiently large:
\begin{equation}
\label{eq:tbb-bounds}
\xia \cdot \invstep{BB1}{k} \le \invstep{}{k}(\tau_k) \le \xib \cdot \lambda_n
\end{equation}
for certain $0 < \xia \le 1$ and $\xib \ge 1$. Consequently, 
\[
(\xib)^{-1} \cdot \lambda_n^{-1} \le \step{}{k}(\tau_k) \le \xia^{-1} \cdot \step{BB1}{k}.
\]

\begin{lemma}
\label{lemma:propertyA}
Let the inverse stepsize $\alpha_{k}(\tau_k)$ satisfy \eqref{eq:tbb-bounds} with $\frac12 < \xia \le 1$ and $\xib \ge 1$. Then such stepsize satisfies Assumption~\ref{assumption}. In addition, given $k$, there exists a constant $\constray \in (0,1)$ such that
\begin{equation} \label{decreasing-gradient}
\vert\gradcomp{\ell+1}{k+1}\vert \le \constray \ \vert\gradcomp{\ell+1}{k}\vert.
\end{equation}
\end{lemma}

\begin{proof}
Part (i) of Assumption~\ref{assumption} immediately follows from the bounds on the BB1 step. Given the hypotheses in (ii) of Assumption~\ref{assumption}, Equation~\eqref{tau-lambda-bb1} and $M_2 = (\xia-\frac12)^{-1}$, it follows that
\begin{align*}
\alpha_{k}(\tau_{k})&\ge \xia \ \frac{\sum_{i=1}^n(\gradcomp{i}{k-1})^2\lambda_i}{\sum_{i=1}^n(\gradcomp{i}{k-1})^2}
\ge \xia \, \lambda_{\ell+1}\, \frac{\sum_{i=\ell+1}^n (\gradcomp{i}{k-1})^2} {\sum_{i=\ell+1}^n(\gradcomp{i}{k-1})^2 + \eps} \\
& \ge \xia\, \lambda_{\ell+1}\, \frac{(\xia-\tfrac12)^{-1}\eps}{(\xia-\tfrac12)^{-1}\eps + \eps} = \frac{\xia}{\xia+\frac12}\, \lambda_{\ell+1}.
\end{align*}
Since $\alpha_{k}(\tau_{k}) \le \xib \, \lambda_n$,
\[
1-\frac{\xia + \tfrac12}{\xia} \le 1 - \lambda_{\ell+1}\, \step{}{k}(\tau_k) \le 1 - \frac{\lambda_{\ell+1}}{\xib \, \lambda_n}.
\]
Given \eqref{recursive-coefs}, this implies $\vert\gradcomp{\ell+1}{k+1}\vert\le \constray \, \vert\gradcomp{\ell+1}{k}\vert$ for constant $\constray := \max \{(2\xia)^{-1}, \, 1 - \lambda_{\ell+1} \, (\xib \, \lambda_n)^{-1} \} < 1$.
\end{proof}

\begin{remark}
\label{remark1} {\rm If $\tau_k = \toneconst \, \invstep{BB2}{k}$ (cf.~\eqref{eq:t1}), Lemma~\ref{lemma:propertyA} does not hold for all the values of $\toneconst$: we must restrict ourselves to $\toneconst > 2$. Nevertheless, in the non-quadratic case, the gradient method is endowed with a line search, where bounds on the stepsize are provided by the user (see, e.g., \cite{raydan1997barzilai} and Section~\ref{sec:nonlin}). In this context, we may try also $\toneconst \le 2$, which corresponds to $\xia <\frac12$.
}
\end{remark}

\subsection{Bounds on gradient components}
We establish two bounds on the gradient components. To do so, the bounds in Assumption~\ref{assumption} are used, with an appropriate restriction on $\xia$. We then show that the first gradient component converges to $0$ under certain conditions.
Let us start with the following lemma, which is an extension of \cite[Lemma 1]{raydan1993barzilai}: this holds for $\xia = \xib = 1$ and it is stated for the components of the error, that enjoy the same recursive formula as the components of the gradient.

\begin{lemma} \label{lemma:first-component}
Under Assumption~\ref{assumption}, $\gradcomp{1}k$ converges to zero Q-linearly, i.e., there exists a constant $\constone \in (0,1)$ such that
\begin{equation}
\label{ineq-1}
\vert\gradcomp{1}{k+1}\vert\le \constone \, \vert\gradcomp{1}k\vert.
\end{equation}
\end{lemma}
\begin{proof}
From part (i) in Assumption~\ref{assumption}, we have
\[
1-\xia^{-1} \le 1 - \lambda_1\,\step{}{k}(\tau_k) \le 1 - (\kappa(A) \, \xib)^{-1},
\]
with $\kappa(A)$ the condition number as in \eqref{condnr}. Thus, when 
applying these bounds to \eqref{recursive-coefs}, we see that \eqref{ineq-1} holds with $\constone = \max\{\xia^{-1}-1, \ 1 - (\kappa(A) \, \xib)^{-1}\}$. The conditions on $\xia$ and $\xib$ guarantee that $\constone$ is indeed in the interval $(0,1)$.
\end{proof}

Unfortunately, it is not possible to prove the $Q$-linear convergence of the other gradient components, but the following inequality will play a role later. This result generalizes \cite[Lemma~2.1]{dai2002r}.
\begin{lemma}
\label{lemma:other-components}
Under Assumption~\ref{assumption}, there exists a constant $\consttwo > 0$ such that for $i = 2, \dots, n$
\begin{equation} \label{ith-error-comp}
\vert\gradcomp{i}{k+1}\vert\le \consttwo \, \vert\gradcomp{i}k\vert.
\end{equation}
\end{lemma}
\begin{proof}
From part (i) in Assumption~\ref{assumption}, we have
\[
1-\xia^{-1} \, \kappa(A) \le 1 - \lambda_i\,\step{}{k}(\tau_k) \le 1 - (\kappa(A) \, \xib)^{-1}.
\]
Application of these bounds to \eqref{recursive-coefs} implies that \eqref{ith-error-comp} holds with positive constant $\consttwo := \max \{\xia^{-1} \, \kappa(A)-1, \ 1-(\kappa(A) \, \xib)^{-1}\}$.
\end{proof}

We note that usually $\consttwo$ will be $\xia^{-1} \, \kappa(A)-1$, and that this quantity may be very large. However, this still provides us with a needed tool for proving the convergence of the gradient method for quadratics. In addition, we remark that Lemma~\ref{lemma:other-components} still holds in general for $\xia \in (0,1]$ and $\xib \ge 1$.

\subsection{Proof of R-linear convergence}
Finally, we are able to state the R-linear convergence result under Assumption~\ref{assumption}, which closely follows the line of the proof of \cite[Thm.~4.1]{dai2003alternate}. The key parts of the proof are the bounds on the gradient components from Lemma~\ref{lemma:first-component} and Lemma~\ref{lemma:other-components}, the result on the $(\ell+1)$st gradient component of Lemma~\ref{lemma:propertyA}. Our contribution is to adapt all these results that were already in \cite{dai2003alternate} but derived from \cite[Property~A]{dai2003alternate}. A slightly less general proof of R-linear convergence was previously presented in \cite{dai2002r}. 

\begin{theorem}
\label{thm:dai}
Let $f$ be a strictly convex quadratic function and let $\bx^\ast = A^{-1}\bb$ be its unique minimizer. Let $\{\bx_k\}$ be the sequence generated by the gradient method where the stepsize satisfies Assumption~\ref{assumption}. Then, either $\bg_k = \zero$ for some finite $k$, or the sequence $\{\|\bg_k\|\}$ converges to zero R-linearly. 
\end{theorem}

\begin{proof} 
Let $G(k,\ell) := \sum_{i=1}^{\ell}(\gradcomp ik)^2$, $\delta_1 = \constone^{\,2}$, $\delta_2 = \consttwo^{\,2}$ and $\delta = \constray^2$. In particular, notice that $G(k,n) = \|\bg_k\|^2$. Assume also that $\consttwo > 1$, otherwise we would immediately conclude the proof due to Lemma~\ref{lemma:other-components}.

{\bf Part I.} First we prove that, for an integer $1 \le \ell \le n-1$ and given $k\ge 1$, if there exists some $\eps_\ell \in (0,M_2^{-1})$ and integer $m_\ell$ such that 
\begin{equation} \label{induction-hp}
G(k+j, \ell) \le \eps_\ell \, \|\bg_k\|^2\quad\text{for}\quad j\ge m_\ell
\end{equation}
then there exists $j_0 \in \{m_\ell, \dots, m_\ell + \Delta_\ell + 1\}$, with $\Delta_\ell = \Delta_\ell(M_2, \eps_\ell, \delta_2, \delta, m_\ell)$, such that
\[
(\gradcomp{\ell+1}{k+j_0})^2\le M_2\, \eps_\ell \, \|\bg_k\|^2.
\] 
Assume that $(\gradcomp{\ell+1}{k+j})^2 > M_2\, \eps_\ell \, \|\bg_k\|^2$ for all $j \in \{m_\ell, \dots, m_\ell+\Delta_\ell\}$. We show that the thesis holds for $j_0 = m_\ell+\Delta_\ell + 1$. First, apply Lemma~\ref{lemma:propertyA} $\Delta_\ell+1$ times, and Lemma~\ref{lemma:other-components} $m_\ell$ times to obtain
\[
(\gradcomp{\ell+1}{k+m_\ell+\Delta_\ell + 1})^2 \, \le \, \delta^{\Delta_\ell+1}\, (\gradcomp{\ell+1}{k+m_\ell})^2 \, \le \, \delta^{\Delta_\ell+1}\, \delta_2^{m_\ell}\, (\gradcomp{\ell+1}k)^2.
\]
Then choose $\Delta_\ell$ as the smallest integer that solves $\delta^{\Delta_\ell+1}\delta_2^{m_\ell} \le M_2\, \eps_\ell$ (such $\Delta_\ell$ exists due to the choice of $\eps_\ell$ and the fact that $\delta < 1 $) and complete the first proof:
\[
(\gradcomp{\ell+1}{k+m_\ell+\Delta_\ell + 1})^2 \le M_2\, \eps_\ell\, \|\bg_k\|^2.
\]
{\bf Part II.} Let $m_{\ell+1} = m_\ell + \Delta_\ell + 1$ and $\eps_{\ell+1} = (1+M_2\, \delta_2^2)\, \eps_\ell$. If \eqref{induction-hp} holds, we show that
\[
G(k+j, \ell+1) \le \eps_{\ell+1}\|\bg_k\|^2\quad\text{for}\quad j\ge m_{\ell+1}.
\]
Since $G(k+j, \ell+1) = G(k+j, \ell) + (\gradcomp{\ell+1}{k+j})^2$ and $m_{\ell+1} > m_\ell$, it is sufficient to prove that
\[
(\gradcomp{\ell+1}{k+j})^2 \le M_2\, \delta_2^2\, \eps_\ell\, \|\bg_k\|^2 \quad \text{for} \quad j\ge m_{\ell+1}.
\]
From the first result, there are infinitely many pairs of indices $j_1,j_2$ with $j_2 \ge j_1 + 2 > j_1 \ge j_0 \ge m_\ell$ such that
\begin{align*}
(\gradcomp{\ell+1}{k+j})^2 &\le M_2\, \eps_\ell \, \|\bg_k\|^2\quad \text{for}\quad j=j_1,j_2, \\
(\gradcomp{\ell+1}{k+j})^2 &> M_2\, \eps_\ell \, \|\bg_k\|^2\quad \text{for}\quad j\in \{j_1+1, \dots, j_2-1\}.
\end{align*}
From Lemma~\ref{lemma:propertyA} it holds that $(\gradcomp{\ell+1}{k+j})^2 \le \delta \, (\gradcomp{\ell+1}{k+j-1})^2 < (\gradcomp{\ell+1}{k+j-1})^2$ for $j \in \{j_1+3, \dots, j_2 + 1\}$, since $\delta < 1$. This results in a chain of inequalities, halting at $j = j_1 + 2$, which corresponds to the rightmost term; to get any further, we apply Lemma~\ref{lemma:other-components}:
\[
(\gradcomp{\ell+1}{k+j})^2 \, \le \, (\gradcomp{\ell+1}{k+j_1+3})^2 \, \le \, (\gradcomp{\ell+1}{k+j_1+2})^2 \, \le \, \delta_2^2\, (\gradcomp{\ell+1}{k+j_1})^2.
\]
Note that the last inequality also holds when $j = j_1,\,j_1+1$ (in place of $j = j_1 + 2$), since we assumed $\delta_2 > 1$. Thus our conclusion is
\[
(\gradcomp{\ell+1}{k+j})^2 \le M_2 \, \delta_2^2 \, \eps_\ell \, \|\bg_k\|^2, \qquad j_1 \le j \le j_2+1.
\]
Since $j_1$ and $j_2$ are chosen arbitrarily and $j_0 \le m_{\ell+1}$, the result automatically holds for any $j \ge m_{\ell+1}$.

{\bf Part III.} Finally, we prove by induction that \eqref{induction-hp} holds for all $1 \le \ell \le n$ with
\[
\eps_\ell = \tfrac14 \, (1+M_2\, \delta_2^2)^{\ell-n}.
\]
For $\ell = 1$ and from Lemma~\ref{lemma:first-component}, the first component of the gradient satisfies $G(k+j, 1) \le \delta_1^j \, \|\bg_k\|^2$. As in the first step, we ask $\delta_1^j \le \eps_1$ and get $j \ge m_1$, with $ m_1 = \lceil \frac{\log \eps_1}{\log \delta_1}\rceil$. Once the thesis is true for some $1 \le \ell \le n-1$, the second step shows that it also holds for $\ell + 1$, with $m_{\ell + 1} = m_\ell + \Delta_\ell + 1$ and $\eps_{\ell+1} = \frac14\, (1 + M_2\, \delta_2^2)\, (1+M_2\, \delta_2^2)^{\ell-n}$. We can conclude that the thesis holds for $\ell = n$, and thus
\[
\|\bg_{k+m_n}\|^2 \le \tfrac{1}{4} \, \|\bg_k\|^2,
\]
where $m_n$ does not depend on $k$. Renumbering the indices as in \cite{dai2002r} enables us to conclude that the $\bg_k$ converge to zero R-linearly.
\end{proof}

We remark that, when the objective function is quadratic, Theorem~\ref{thm:dai} shows that no line search is required to guarantee the convergence of the gradient method with TBB stepsizes.
For generic unconstrained optimization problems, we add a line search procedure with a condition of \emph{sufficient decrease} in the next section.

\section{Generic nonlinear functions} \label{sec:nonlin}
We now turn our attention to generic (non-quadratic) continuous differentiable functions $f$.
As the expression \eqref{hss} is not suitable since the Hessian is usually not available, we use the generalization \eqref{hssg}. Just as for the quadratic case (see Proposition~\ref{prop:beta}), the well-known BB1 and BB2 stepsizes are retrieved for $\tau_k \to \pm \infty$ and $\tau_k = 0$, respectively. 
In the quadratic case, the secant equation $\by_{k-1} = B_k \bs_{k-1}$ holds for $B_k \equiv A$, and this allows for the interpretation of BB steps and the TBB step as Rayleigh quotients of $A$.
When $f$ is a generic function, the average Hessian $B_k = \int_0^1\nabla^2f(\bx_{k-1} + t \, \bs_{k-1})\,dt$ satisfies the secant equation (cf., e.g., \cite[Eq.~(6.11)]{nocedal2006numerical}), and thus we can still think at BB steps and the TBB step as Rayleigh quotients, which approximate the eigenvalues of this average $B_k$ instead of the ones of $\nabla^2f(\bx_k)$.
We note that, under the condition that $B_k$ is SPD, all results of Sections~\ref{sec:frame} and \ref{sec:abb} continue to hold for generic functions when replacing $A$ by $B_k$.

Algorithm~\ref{algo:nonlin} shows a pseudocode for TBB-step methods for general nonlinear unconstrained optimization problems.
As usual (cf., e.g., \cite[Alg.~1]{daniela2018steplength}), unlike the quadratic case of Algorithm~\ref{algo:tbb}, safeguards are added for the steplength.
We also include the nonmonotone line search strategy from \cite{grippo1986nonmonotone} (see Line~2). Line~3 features a well-known condition of sufficient decrease, where we take the common values of the line search parameters $c_{\rm{ls}} = 10^{-4}$, $\sigma_{\rm{ls}} = \frac12$ (cf.~\cite[p.~33]{nocedal2006numerical}), and $M = 10$ in the experiments. One of the key ingredients to prove the global convergence of Algorithm~\ref{algo:nonlin} is the existence of uniform bounds on the stepsize, i.e., $\step{}{k}\in [\beta_{\min}, \beta_{\max}]$ for all $k$. Since the TBB stepsize \eqref{hssg} with safeguard lies in this interval, the convergence of the algorithm is guaranteed by \cite[Thm.~2.1]{raydan1997barzilai}. It is possible to show R-linear convergence of the algorithm for uniformly convex functions (cf.~\cite[Thm.~3.1, Eq.~(31)]{dai2002nonmonotone}). We remark that the initial starting steplength in the nonmonotone line search is different from the algorithm in \cite{raydan1997barzilai}, and this might lead to a smaller number of backtracking steps. 

The last two features that may significantly affect the speed of the algorithm are the initial stepsize, and the treatment of uphill directions, or, equivalently, negative steplengths. Popular choices for the initial stepsize are $\step{}{0} = 1$ (cf., e.g., \cite{raydan1997barzilai, daniela2018steplength}) or $\step{}{0} = \|\bg_0\|^{-1}$ (cf., e.g., \cite{dai2003alternate}), where the norm is the Euclidean norm or the $\infty$-norm. Line~8 deals with possible uphill directions: \gf{when $\bs_{k-1}^T\by_{k-1} < 0$, $B_k$ is not SPD, and thus all the properties studied in Sections~\ref{sec:frame}--\ref{sec:conv} do not necessarily hold. The TBB step may still be positive for some target values, but does not have a clear connection to the eigenvalues of $B_k$. In fact, when $\bs_{k-1}^T\by_{k-1} < 0$, the TBB steps (including the BB steps) render a negative approximation of the inverse eigenvalues of $B_k$.} Therefore, the tentative $\beta_{k}$ is replaced by a certain $\altstep_{k} > 0$. A possible choice is $\altstep_{k} \equiv \beta_{\max}$ (see, e.g., \cite{daniela2018steplength}), but this stepsize may be huge and might cause overflow problems. Raydan \cite{raydan1997barzilai} proposes to set $\altstep_{k} = \max(\min(\|\bg_k\|_2^{-1}, \, 10^{5}), \ 1)$, which is an attempt to move away from the uphill direction, while keeping $\|\altstep_k\, \bg_k\|$ moderate. 
Others (e.g., \cite{dai2006cyclic, huang2022acceleration}) simply use $\altstep_k = \|\bg_k\|^{-1}$, as it is done for the first stepsize. There is also an interesting alternative of \cite{park2020variable} that reuses the previous steplength $\altstep_k = \step{}{k-1}$; this strategy resembles the cyclic gradient method, where the same BB stepsize is reused for several iterations \cite{dai2006cyclic}. 

\begin{algorithm}
\caption{A TBB method for general nonlinear functions}
\label{algo:nonlin}
{\bf Input}: continuous differentiable function $f$, initial guess $\bx_0$, initial stepsize $\beta_0 > 0$, tolerance {\sf tol}; safeguarding parameters $\beta_{\max} > \beta_{\min} > 0$; line search parameters $c_{\rm{ls}}$, $\sigma_{\rm{ls}} \in (0,1)$; memory integer $M>0$; replacement strategy for negative stepsizes $\altstep_k > 0$ \\
{\bf Output}: approximation to minimizer $\argmin_{\bx} f(\bx)$ \\
\begin{tabular}{rl}
{\footnotesize 1}: & Set $\bg_0 = \nabla f(\bx_0)$ \\
& {\bf for} $k = 0, 1, \dots$ \\
{\footnotesize 2}: & \phantom{M} $\nu_k = \beta_k$, \ $f_{\text{ref}} = \max \, \{ \, f(\bx_{k-j}) \, : \, 0 \le j \le \min(k,M-1) \, \}$ \\
{\footnotesize 3}: & \phantom{M} {\bf while} \ $f(\bx_k-\nu_k \, \bg_k) > f_{\text{ref}} - c_{\rm{ls}} \, \nu_k \, \|\bg_k\|^2$ \ {\bf do} \ $\nu_k = \sigma_{\rm{ls}} \, \nu_k$ \ {\bf end} \\
{\footnotesize 4}: & \phantom{M} Set \ $\bs_k = -\nu_k \, \bg_k$ \ and update \ $\bx_{k+1} = \bx_k + \bs_k$ \\
{\footnotesize 5}: & \phantom{M} Compute the gradient \ $\bg_{k+1} = \nabla f(\bx_{k+1})$ \\
{\footnotesize 6}: & \phantom{M} {\bf if} \ $\|\bg_{k+1}\| \le {\sf tol} \cdot \|\bg_0\|$, \ {\bf return}, \ {\bf end} \\
{\footnotesize 7}: & \phantom{M} $\by_k = \bg_{k+1} - \bg_k$ \\
{\footnotesize 8}: & \phantom{M} \gf{{\bf if} \ $\bs_k^T\by_k < 0$, set $\step{}{k+1} = \altstep_{k+1}$}\\
{\footnotesize 9}: & \phantom{M} \gf{{\bf else}} \ choose $\tau_{k+1}$, compute the TBB step $\beta_{k+1}(\tau_{k+1})$ \eqref{hssg} \ \gf{\bf end}\\
{\footnotesize 10}: & \phantom{M} Set $\beta_{k+1} = \min(\max(\beta_{k+1}, \, \beta_{\min}), \ \beta_{\max})$\\
\end{tabular}
\end{algorithm}

Algorithms~\ref{algo:tbb} and \ref{algo:nonlin} can be combined with preconditioning or scaling.
Scaling may be viewed as the simplest case of preconditioning, that is, by a diagonal SPD matrix. Scaling is a powerful and efficient technique;
we refer to \cite{luengo2002preconditioned} for scaling techniques for unconstrained optimization problems.
The combination of preconditioning and BB steps for quadratic problems has been discussed in \cite{molina1996preconditioned}.
The use of scaling or more general preconditioning is outside the scope of this paper.

\section{Numerical experiments} \label{sec:exp}
We test various target strategies on strictly convex quadratics (Algorithm~\ref{algo:tbb}) and generic differentiable functions (Algorithm~\ref{algo:nonlin}).
The purpose of these experiments is to show numerically that the introduction of an adaptive target in the stepsizes \eqref{hss} can sometimes lead to better convergence results, in terms of number of iterations and function evaluations. For non-convex problems, we also observe that different targets can sometimes detect different local optima. 

\subsection{Sweeping the spectrum of the Hessian matrix} 
Given the recursive definition of the gradient \eqref{recursive-err} for quadratic problems, one can see that if a stepsize sweeps the spectrum of the Hessian matrix appropriately, then the convergence of the corresponding gradient method is faster. Before moving to a detailed analysis of the performances of different stepsizes, we illustrate the sweeping capability of each steplength on the three quadratic problems proposed by \cite{daniela2018steplength}, QP1, QP2 and QP3, in the same setting as \cite{daniela2018steplength}. All three problems have a diagonal Hessian with eigenvalues $0 < \lambda_1 \le \dots \le \lambda_n$. The eigenvalues of QP1 follow the asymptotic distribution of the eigenvalues of a class of covariance matrices; those of QP2 are such that the ratio between two consecutive eigenvalues is constant. In the last problem, eigenvalues are clustered in two groups; see \cite{daniela2018steplength} for further details. Table~\ref{tab:experiments} reports all the implemented target strategies, divided into three groups: known schemes from the literature, positive targets, and (negative) targets inspired by the cotangent function. The stepsizes studied throughout this section are summarized in Table~\ref{tab:experiments}.
\begin{table}[htb!]
\centering
\caption{Strategies for the stepsize.}
{\footnotesize \begin{tabular}{lll}
\toprule
Method & Target $\tau_k$ & Reference \\
\midrule
BB1 & $\pm\infty$ & Cf.~\cite{bb1988}\\
BB2 & $0$ & Cf.~\cite{bb1988}\\
ABB & (NA) & Cf.~\cite{zhou2006abb}, $\eta = 0.8$\\
$\text{ABB}_{\text{min}}$ & (NA) & Cf.~\cite{frassoldati2008new}, $\eta = 0.8$, $m = 4$\\
$\text{ABB}_{\text{bon}}$ & (NA) & Cf.~\cite{bonettini2009scaled}, $\eta_0 = 0.5$, $m = 4$\\
[1.5mm]
IBB2 2.01 & $2.01\,\invstep{BB2}{k}$& Eq.~\eqref{eq:t1}\\[0.5mm]
IBB2 100 & $100\,\invstep{BB2}{k}$ & Eq.~\eqref{eq:t1}\\[0.5mm]
ITER & $k\,\invstep{BB2}{k}$& Eq.~\eqref{eq:t2}\\[1.5mm]
COT~11 & $-\cot(\bs_{k-1}, \by_{k-1})$ & Eq.~\eqref{cot} \\
COT~H1 & $-\cos^{1/2}(\bs_{k-1},\by_{k-1})\,/\,\sin(\bs_{k-1}, \by_{k-1})$ & Eq.~\eqref{cot}\\
COT~1H & $-\cos(\bs_{k-1},\by_{k-1})\,/\,\sin^{1/2}(\bs_{k-1}, \by_{k-1})$ & Eq.~\eqref{cot}\\
COT~21 & $-\cos^2(\bs_{k-1},\by_{k-1})\,/\,\sin(\bs_{k-1}, \by_{k-1})$ & Eq.~\eqref{cot}\\
COT~12 & $-\cos(\bs_{k-1},\by_{k-1})\,/\,\sin^2(\bs_{k-1}, \by_{k-1})$ & Eq.~\eqref{cot}\\
\bottomrule
\end{tabular}}
\label{tab:experiments}
\end{table}

We also consider two effective variants of ABB, which we indicate with $\text{ABB}_{\text{min}}$ \cite{frassoldati2008new} and $\text{ABB}_{\text{bon}}$ \cite{bonettini2009scaled}. 
In the first one, we take the smallest BB2 stepsize over the last $m+1$ iterations, when $\cos^2(\bs_{k-1},\by_{k-1})$ is small:
\[
\step{ABB_{\rm min}}k = \left\{
\begin{array}{ll}
\min\{\step{BB2}j\mid j = \max\{1,k-m\},\dots,k\}, & \ \text{if} \ \step{BB2}k < \eta \, \step{BB1}k, \\[1.5mm]
\step{BB1}k, & \ \text{otherwise.}
\end{array}
\right.
\]
$\text{ABB}_{\text{bon}}$ is defined in the same way as $\text{ABB}_{\text{min}}$ but with an adaptive threshold $\eta$: starting from $\eta_0 = 0.5$, this is updated as
\[
\eta_{k+1} = \left\{
\begin{array}{ll}
0.9\,\eta_k, & \ \text{if} \ \step{BB2}k < \eta_k \, \step{BB1}k, \\[0.5mm]
1.1\,\eta_k, & \ \text{otherwise.}
\end{array}
\right.
\]

With respect to the positive targets, Remark~\ref{remark1} suggests the use of $\rho > 2$ in \eqref{eq:t1} to ensure the convergence of the corresponding gradient method for quadratic functions. The aim of setting $\rho = 2.01$ is to stay close to this lower bound and take the largest possible stepsizes (which are larger than $\betaa_k$). An approach with $\rho = 100$ picks positive targets $\tau_k$ such that the corresponding stepsize $\step{}{k}$ is close to, but still larger than $\betaa_k$.

Figure \ref{fig:qps} shows the inverse stepsize value $\alpha_k$ through the iterations for QP1, QP2 and QP3. For each problem, we select the four stepsizes that require the smallest number of iterations for the convergence. All stepsizes seem to explore the whole spectrum in all problems, but in different ways. The two variants of ABB, $\text{ABB}_{\text{min}}$ and $\text{ABB}_{\text{bon}}$, perform well in all three problems. As a result of their definition, they both tend to recycle the same stepsize for some consecutive iterations. Interestingly, they gradually cancel the gradient components corresponding to the largest eigenvalues of the Hessian. This feature is particularly clear in QP2: since some gradient components are annihilated at some stage, in the latest iterations the stepsizes are concentrated only in the eigenvalues corresponding to the remaining gradient components. The stepsize COT~H1 to some extent shows the same behavior in QP2.

Other stepsize strategies do not explicitly remove some of the gradient components in an early phase, but still show a comparable number of iterations.

\begin{figure}[htb!]
\centering
\subfloat[QP1]{\includegraphics[width=\textwidth]{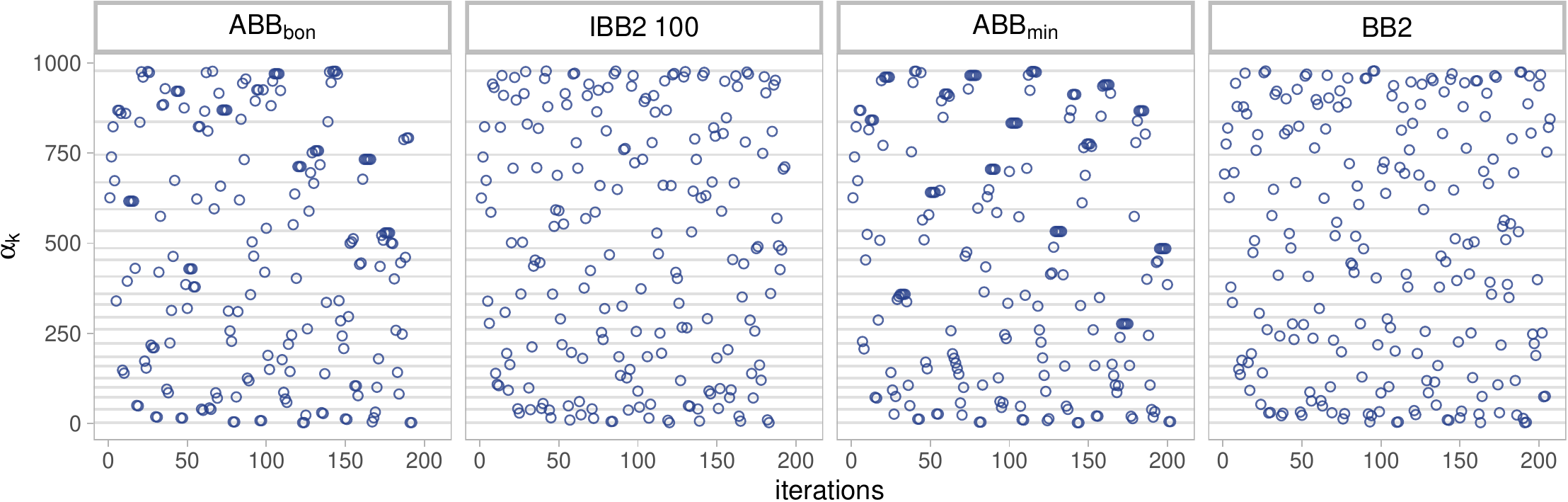}}\hfill
\subfloat[QP2]{\includegraphics[width=\textwidth]{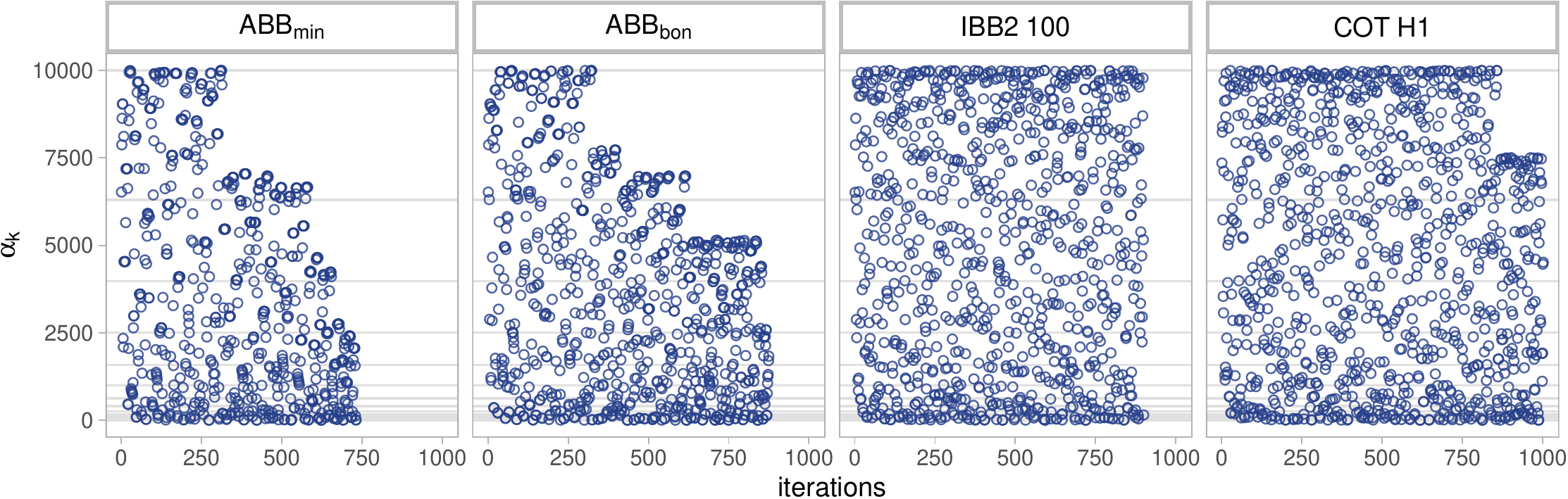}}\hfill
\subfloat[QP3]{\includegraphics[width=\textwidth]{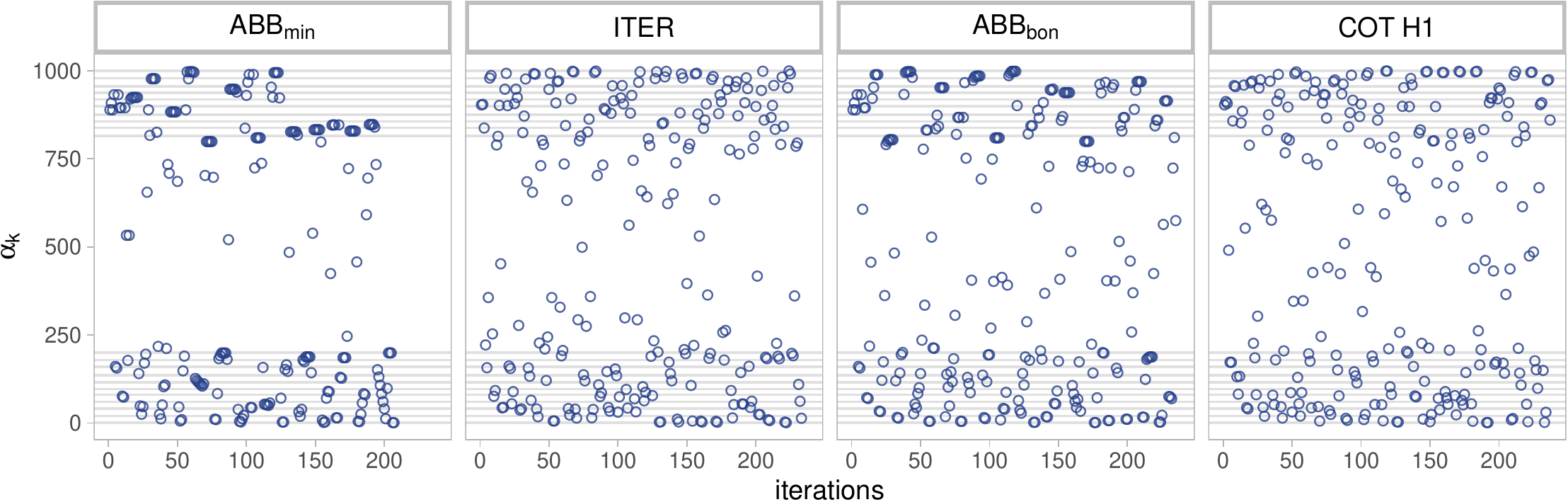}}\hfill
\caption{Inverse stepsize $\alpha_k$ per iteration. Gray lines correspond to 21 equally spaced eigenvalues of the Hessian, including the smallest and the largest, for the three quadratic problems QP1--QP3. Stepsizes are ordered based on the (increasing) number of iterations of the corresponding gradient method.}
\label{fig:qps}
\end{figure}

\subsection{Strictly convex quadratic functions}
For the problems of the form \eqref{quad}, we take examples from the SuiteSparse Matrix Collection \cite{suitesparse}. The selected matrices $A$ are 65 symmetric positive definite matrices with a number of rows between $10^2$ and $10^4$, and an estimated condition number $\le 10^{8}$ (the condition number is estimated via the routine {\sf condest} in the {\sf Matrix} R package). 
The vector $\bb$ is chosen so that the solution of $A\bx = \bb$ is $\bx^\ast = \be$, the vector of all ones.
For all problems, the starting vector is $\bx_0 = -10\,\be$, and the initial stepsize is $\step{}{0} = 1$. The algorithm stops when $\|\bg_k\| \le {\sf tol}\,\|\bg_0\|$ with ${\sf tol} = 10^{-6}$, or when $5\cdot 10^4$ iterations are reached. The problem {\sf nos4} is scaled by the Euclidean norm of the first gradient.

We compare the performances of the different stepsizes in Table~\ref{tab:experiments} by means of a performance profile \cite{perfprofile}. The cost of solving each problem is normalized based on the minimum cost for that problem, to get the \emph{performance ratio} \cite{perfprofile}. The most efficient method solves the given problem with performance ratio $1$, while all other methods solve it with a performance ratio at least $1$. We plot the ratio of problems solved by a method within a certain factor of the smallest cost; this results in a cumulative distribution for each method. The algorithms are rated based on the maximum cost that one is willing to pay to get convergence. An infinite cost is assigned whenever a method is not able to solve a problem to the tolerance within the maximum number of iterations.

Given a problem, the cost of each gradient method differs only in the computation of the stepsize. For the computation of \eqref{hss}, we exploit the fact that $\bs_k = -\beta_k\bg_k$ (cf.~Algorithm~\ref{algo:tbb}) and that the stopping criterion is based on $\|\bg_k\|$. The TBB step requires the additional computation of $\bg_k^T\by_k$ and $\|\by_k\|$ and therefore takes two extra inner products. ABB, the ABB variants and BB2 need the same quantities, while BB1 is slightly less expensive, with only one extra inner product.
Moreover, computing a gradient is usually (much) more expensive than determining the stepsize.
As a consequence, the various methods have approximately the same computational cost per iteration. In addition, in the quadratic case we do not employ a line search procedure, thus we take the number of iterations as the basis of our performance profile.

\begin{figure}[htb!]
\centering \includegraphics[width=0.7\textwidth]{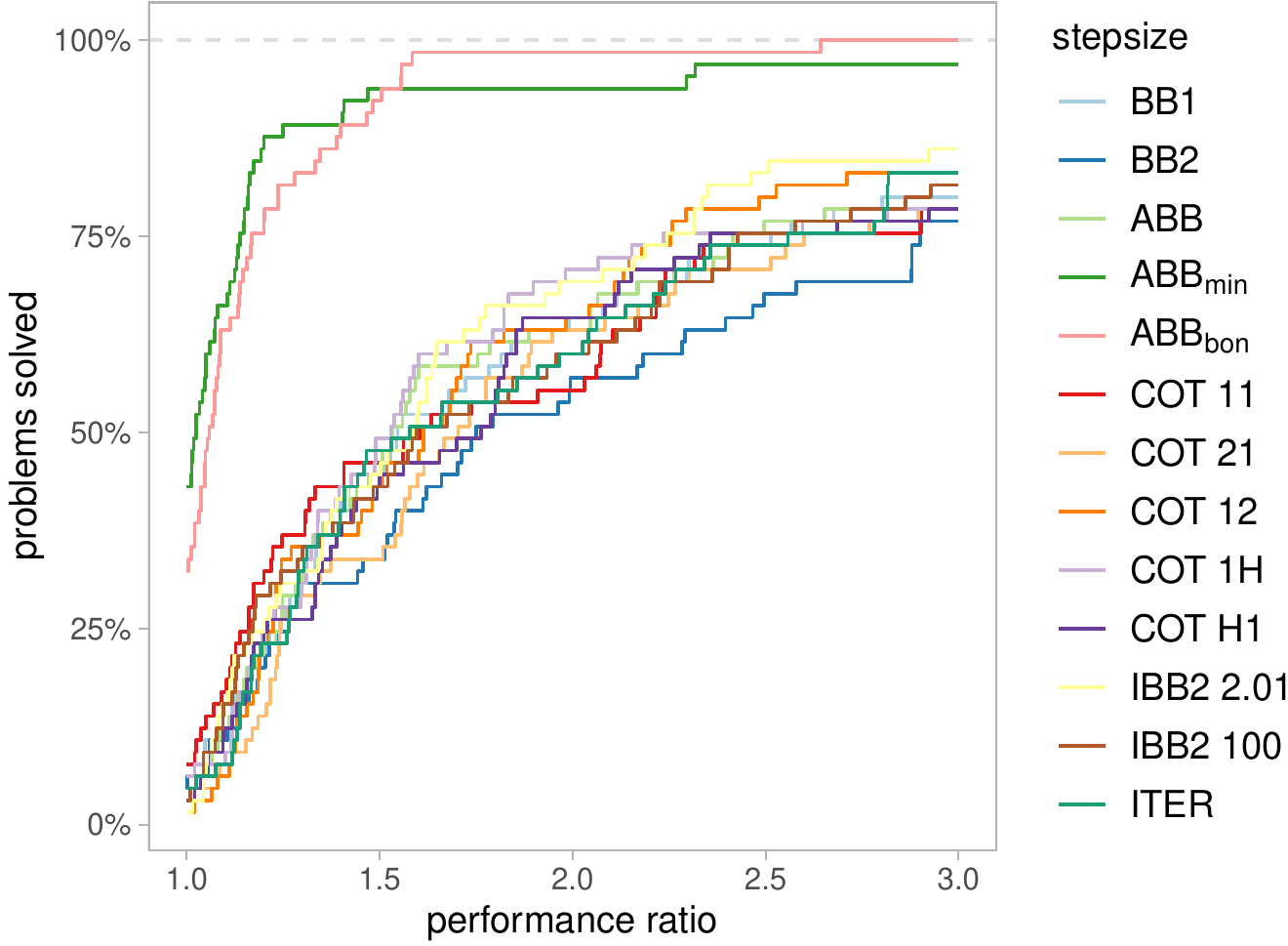}
\caption{Performance profile for strictly convex quadratic problems, based on the number of iterations.}
\label{fig:perfprofquad}
\end{figure}

Figure~\ref{fig:perfprofquad} displays the performance profile of the different stepsizes, based on the number of iterations. The performance ratio is considered in the interval $[1,\, 3]$. It is evident that the ABB variants perform much better than all other stepsizes. The stepsizes proposed in this paper are comparable with BB1, BB2, and ABB. One reason for this behavior may be that the ABB variants are especially favorable to quadratic problems: in fact, when the Hessian matrix is constant, previous BB2 stepsizes still approximate some eigenvalue of the current Hessian. This is no longer true for the generic problems: as a consequence, we will see that, in that case, the ABB variants have similar performances compared to the other stepsizes.

The COT~11 step for a performance ratio $\le 1.5$ and the IBB2~2.01 step for $[1.5,\,3]$ perform better than the other TBB steplengths; the latter one competes with COT~1H, and COT~12 in different segments of its interval. Nevertheless, while IBB2~2.01 solves slightly more than $80\,\%$ of the problems with a performance ratio $\le 3$, $\text{ABB}_{\text{bon}}$ manages to solve all problems within the same range.

Since we study the performance ratio in a restricted range, we also collect some summary information in Table~\ref{tab:quad1}. It is interesting to notice that the minimum performance ratio is $1$ for all stepsizes: this means that for each stepsize there is at least one problem where that stepsize performs at least as well as the others. The ABB variants appear to be more robust than the other stepsizes, followed by BB1 and the class of TBB steps with positive target. The steplength $\text{ABB}_{\text{min}}$ solves the highest proportion of problems at the lowest cost. 
\begin{table}[htb!]
\centering
\caption{Gradient method with TBB steps and ABB variants: proportion of solved problems within $5\cdot 10^4$ iterations, proportion of problems solved with unit cost (i.e., performance ratio (PR) equal to 1), average and standard deviation of the performance ratios (per method), and range of the performance ratios.}
\label{tab:quad1}
{\footnotesize
\begin{tabular}{lrrccc}
 \toprule
Stepsize & Solved & $\text{PR}=1$ & Avg & Sd & Range \\ 
 \midrule
$\text{ABB}_\text{bon}$ & 100\,\% & 32.3\,\% & 1.15 & 0.25 & $[1.00,\,\ph{1}2.64]$ \\ 
$\text{ABB}_\text{min}$ & 97\,\% & 43.1\,\% & 1.11 & 0.24 & $[1.00,\,\ph{1}2.32]$ \\ 
 ITER & 94\,\% & 4.6\,\% & 2.16 & 1.83 & $[1.00,\,12.56]$ \\ 
 BB1 & 94\,\% & 3.1\,\% & 2.11 & 1.75 & $[1.00,\,11.52]$ \\ 
 IBB2 100 & 94\,\% & 3.1\,\% & 2.14 & 1.95 & $[1.00,\,14.48]$ \\ 
 IBB2 2.01 & 94\,\% & 1.5\,\% & 2.04 & 2.25 & $[1.00,\,17.49]$ \\ 
 BB2 & 92\,\% & 4.6\,\% & 2.31 & 2.01 & $[1.00,\,14.07]$ \\ 
 COT H1 & 92\,\% & 3.1\,\% & 2.15 & 1.88 & $[1.00,\,13.51]$ \\ 
 COT 11 & 91\,\% & 7.7\,\% & 2.06 & 1.68 & $[1.00,\,11.27]$ \\ 
 COT 1H & 91\,\% & 6.2\,\% & 2.03 & 1.68 & $[1.00,\,10.38]$ \\ 
 COT 12 & 91\,\% & 1.5\,\% & 1.90 & 1.39 & $[1.00,\,10.05]$ \\ 
 ABB & 89\,\% & 4.6\,\% & 2.12 & 2.10 & $[1.00,\,14.54]$ \\ 
 COT 21 & 89\,\% & 3.1\,\% & 2.18 & 1.82 & $[1.00,\,11.18]$ \\ 
\bottomrule
\end{tabular}
}
\end{table}

\subsection{Unconstrained optimization}
We take some generic differentiable functions from the collections in \cite{andrei2008unconstrained, more1981testing, raydan1997barzilai} and the suggested starting points $\bx_0$ therein, as listed in Table~\ref{tab:unconstrained}. For Griewank's function we choose $\bx_0 = \be$. For each problem, we also consider the starting points $5\,\bx_0$ and $10\,\bx_0$, in line with \cite{more1981testing}.
\begin{table}[htb!]
\centering
\caption{Unconstrained optimization test problems.}
{\footnotesize
\begin{tabular}{lclc}
 \toprule
Name & Reference & Name & Reference\\ 
 \midrule
{\sf Broyden tridiagonal} & \cite{more1981testing} & {\sf Extended White and Holst} & \cite{andrei2008unconstrained} \\ 
 {\sf Diagonal 1} & \cite{andrei2008unconstrained} & {\sf Full Hessian FH1} & \cite{andrei2008unconstrained} \\ 
 {\sf Diagonal 2} & \cite{andrei2008unconstrained} & {\sf Full Hessian FH2} & \cite{andrei2008unconstrained} \\ 
 {\sf Diagonal 3} & \cite{andrei2008unconstrained} & {\sf Generalized Rosenbrock} & \cite{andrei2008unconstrained} \\ 
 {\sf Diagonal 4} & \cite{andrei2008unconstrained} & {\sf Generalized tridiagonal 1} & \cite{andrei2008unconstrained} \\ 
 {\sf Extended Beale} & \cite{andrei2008unconstrained} & {\sf Generalized tridiagonal 2} & \cite{andrei2008unconstrained} \\ 
 {\sf Extended Freudenstein and Roth} & \cite{andrei2008unconstrained} & {\sf Generalized White and Holst} & \cite{andrei2008unconstrained} \\ 
 {\sf Extended Himmelblau} & \cite{andrei2008unconstrained} & {\sf Griewank} & \cite{griewank1981generalized} \\ 
 {\sf Extended Powell} & \cite{more1981testing} & {\sf Hager} & \cite{andrei2008unconstrained} \\ 
 {\sf Extended PSC1} & \cite{andrei2008unconstrained} & {\sf Perturbed quadratic} & \cite{andrei2008unconstrained} \\ 
 {\sf Extended Rosenbrock} & \cite{more1981testing} & {\sf Strictly Convex 1} & \cite{raydan1997barzilai} \\ 
 {\sf Extended TET} & \cite{andrei2008unconstrained} & {\sf Strictly Convex 2} & \cite{raydan1997barzilai} \\ 
 {\sf Extended tridiagonal 1} & \cite{andrei2008unconstrained} & {\sf Trigonometric} & \cite{more1981testing} \\ 
\bottomrule
\end{tabular}}
\label{tab:unconstrained}
\end{table}

For all the test functions, we pick $n = 100$ variables. The {\sf generalized Rosenbrock}, {\sf generalized White and Holst} and {\sf extended Powell} objective functions have been scaled by the Euclidean norm of the first gradient. 

The gradient method for unconstrained optimization problems requires the tuning of more parameters than the gradient method for quadratic functions. We maintain the choices made in \cite{daniela2018steplength} and set $\beta_{\min} = 10^{-30}$, $\beta_{\max} = 10^{30}$, $c_{\rm{ls}} = 10^{-4}$, $\sigma_{\rm{ls}} = \frac12$, $M = 10$, and $\step{}{0} = 1$. Although one might argue that the bounds on the stepsize are extremely large, the aim of this choice is to accept the BB stepsize as frequently as possible. Following Raydan \cite{raydan1997barzilai}, we choose $\altstep_k = \max(\min(\|\bg_k\|_2^{-1}, \, 10^{5}), \ 1)$ as the replacement for negative stepsizes. Although the alternative of recycling the last positive stepsize also seems plausible, in our experiments we find that this strategy may lead to poor performance for some stepsizes. Raydan's rule seems to behave well in combination with all stepsizes. Again the algorithm stops when $\|\bg_k\| \le {\sf tol}\,\|\bg_0\|$, or when $5\cdot 10^4$ iterations are reached. We show the results for three levels of tolerance ${\sf tol} \in \{10^{-4},\,10^{-6},\,10^{-8}\}$. All different steps in Table~\ref{tab:experiments} are tested.

Since some test problems are non-convex, we check whether all gradient methods converged to the same stationary point for different stepsizes. For this reason, the following analysis will not include {\sf Broyden tridiagonal}, {\sf extended Freudenstein and Roth}, {\sf generalized tridiagonal 2}, {\sf Griewank}, {\sf trigonometric}. 

\begin{remark}
\label{remark2}
Aside the computational cost of an algorithm, the quality of the reached minimum is also an important aspect. In this context, it is interesting to notice that in the {\sf extended Freudenstein and Roth} function, for the setting ${\sf tol} = 10^{-8}$ and starting point $\bx_0$, the choice of IBB2~2.01 leads to the global optimum $f = 0$, while all the other gradient methods converge to $f \approx 1225$. The convergence of IBB2~2.01 takes approximately seven times the number of function evaluations of the fastest method, but the gradient method finds a better solution.
\end{remark}

As the performance profile, we may consider two different costs: the number of function evaluations and the number of iterations. The latter corresponds to the number of gradient evaluations, since the line search in Algorithm~\ref{algo:nonlin} does not require the computation of the gradient at the new tentative iterate. Our comparison is on the number of function evaluations, since this is the dominant cost for our test cases.
The performance profiles are shown in Figure~\ref{fig:perfunconstr} in the range $[1,\,3]$, for the number of function evaluations, and various tolerances and starting points. 
\begin{figure}[htb!]
\centering \includegraphics[width=\textwidth]{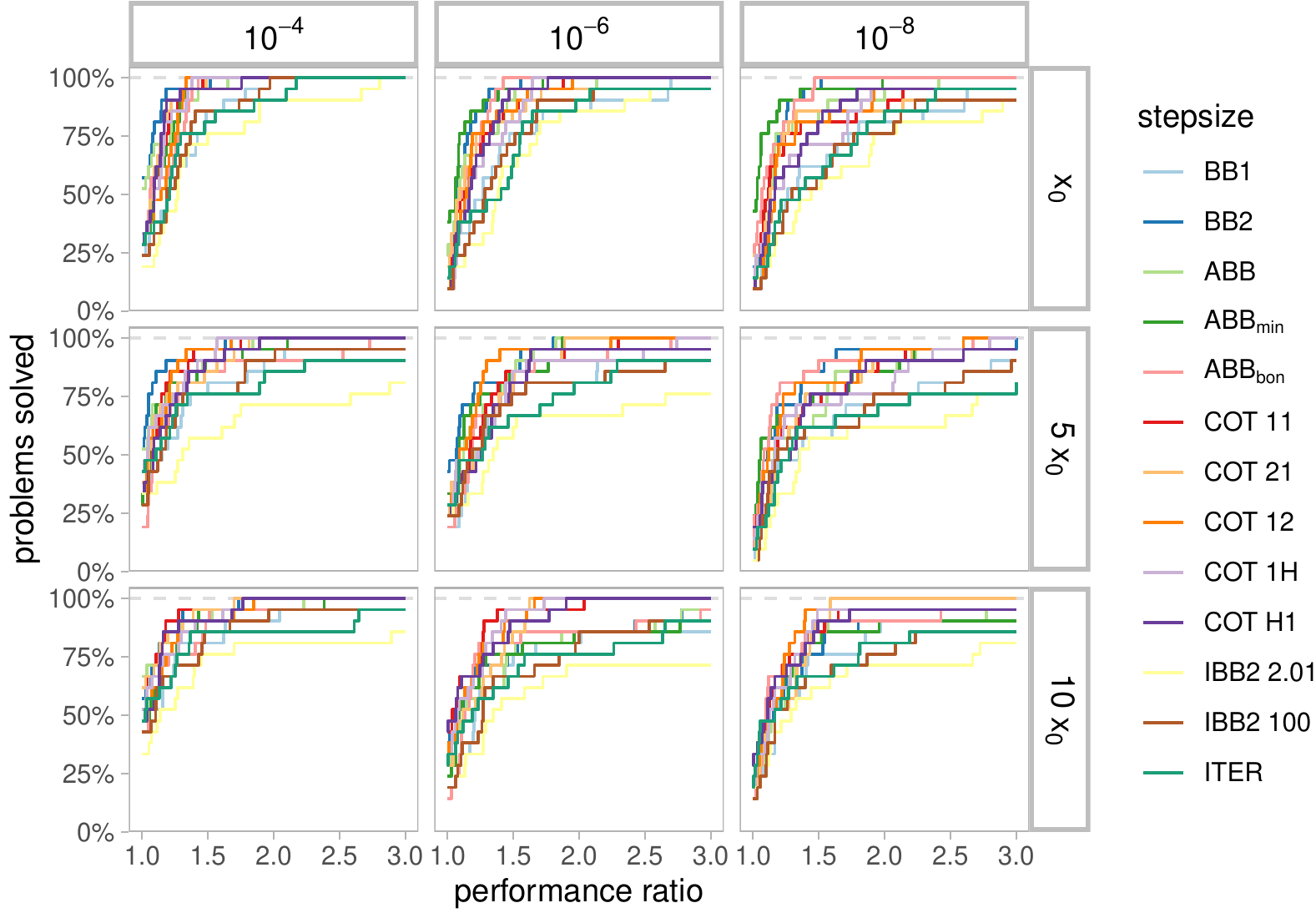}
\caption{Performance profiles for generic unconstrained optimization problems, based on the number of function evaluations. Each set of performance profiles considers a different $\tol \in \{10^{-4},\,10^{-6},\,10^{-8}\}$ (columns) and starting point $\{\bx_0,\,5\,\bx_0,\,10\,\bx_0\}$ (rows).}
\label{fig:perfunconstr}
\end{figure}

As the tolerance decreases, from left to right, we notice that the performance profiles become more distinct. In contrast with the performance profile for the quadratic case, here the ABB variants are not as prominent as before. We can still distinguish their curves, along with the one of BB2, when ${\sf tol} = 10^{-8}$ and the starting points are either $\bx_0$ or $5\,\bx_0$. The situation changes when the starting point is $10\,\bx_0$: especially when ${\sf tol} = 10^{-6}$, the ABB variants are mixed with some curves of the cotangent family when the performance ratio is $\le 1.5$. Then all curves of the cotangent family display more favorable behavior than the rest of the curves. 

\begin{table}[htb!]
\centering
\caption{Gradient method with nonmonotone line search: proportion of solved problems within $5\cdot 10^4$ iterations, proportion of problems solved with unit cost, i.e., performance ratio (PR) equal to 1, average and standard deviation of the performance ratios, range of the performance ratios. The performance ratio is based on the number of function evaluations.}
\label{tab:nfeunconstr}
\subfloat[${\sf tol} = 10^{-6}$, starting point $\bx_0$]{
\scalebox{0.7}{
\begin{tabular}{lrcccc}
\toprule
Stepsize & Solved & $\text{PR}=1$ & Avg & Sd & Range \\
 \midrule
 $\text{ABB}_{\text{min}}$ & 100\% & 38.1\% & 1.11 & 0.17 & $[1.00,\,1.73]$ \\ 
 BB2 & 100\% & 23.8\% & 1.12 & 0.14 & $[1.00,\,1.56]$ \\ 
 ABB & 100\% & 23.8\% & 1.19 & 0.27 & $[1.00,\,2.14]$ \\ 
 $\text{ABB}_{\text{bon}}$ & 100\% & 19.0\% & 1.14 & 0.13 & $[1.00,\,1.42]$ \\ 
 COT 21 & 100\% & 19.0\% & 1.27 & 0.52 & $[1.00,\,3.38]$ \\ 
 COT 11 & 100\% & 14.3\% & 1.21 & 0.24 & $[1.00,\,1.88]$ \\ 
 IBB2 2.01 & 100\% & 14.3\% & 1.55 & 0.60 & $[1.00,\,3.49]$ \\ 
 ITER & 100\% & 14.3\% & 1.44 & 0.51 & $[1.00,\,3.17]$ \\ 
 BB1 & 100\% & \ph{1}9.5\% & 1.42 & 0.50 & $[1.00,\,2.70]$ \\ 
 COT 12 & 100\% & \ph{1}9.5\% & 1.21 & 0.23 & $[1.00,\,1.95]$ \\ 
 COT 1H & 100\% & \ph{1}9.5\% & 1.24 & 0.21 & $[1.00,\,1.64]$ \\ 
 COT H1 & 100\% & \ph{1}9.5\% & 1.21 & 0.19 & $[1.00,\,1.76]$ \\ 
 IBB2 100 & 100\% & \ph{1}9.5\% & 1.43 & 0.52 & $[1.00,\,3.35]$ \\
\bottomrule
\end{tabular}}
}
\subfloat[${\sf tol} = 10^{-6}$, starting point $5\,\bx_0$]{
\scalebox{0.7}{
\begin{tabular}{lrcccc}
\toprule
Stepsize & Solved & $\text{PR}=1$ & Avg & Sd & Range \\
 \midrule
 BB2 & 100\% & 43\% & 1.15 & 0.23 & $[1.00,\,\ph{1}1.80]$ \\ 
 $\text{ABB}_{\text{min}}$ & 100\% & 33\% & 1.21 & 0.29 & $[1.00,\,\ph{1}1.87]$ \\ 
 COT 12 & 100\% & 29\% & 1.18 & 0.27 & $[1.00,\,\ph{1}2.24]$ \\ 
 ITER & 100\% & 29\% & 1.93 & 2.12 & $[1.00,\,10.52]$ \\ 
 COT 11 & 100\% & 24\% & 1.26 & 0.33 & $[1.00,\,\ph{1}2.30]$ \\ 
 COT 21 & 100\% & 24\% & 1.26 & 0.28 & $[1.00,\,\ph{1}1.89]$ \\ 
 COT 1H & 100\% & 24\% & 1.33 & 0.47 & $[1.00,\,\ph{1}2.74]$ \\ 
 COT H1 & 100\% & 24\% & 1.42 & 0.84 & $[1.00,\,\ph{1}4.94]$ \\ 
 IBB2 2.01 & 100\% & 24\% & 2.61 & 3.79 & $[1.00,\,18.59]$ \\ 
 IBB2 100 & 100\% & 24\% & 1.58 & 0.94 & $[1.00,\,\ph{1}4.27]$ \\ 
 BB1 & 100\% & 19\% & 1.53 & 0.89 & $[1.00,\,\ph{1}4.74]$ \\ 
 ABB & 100\% & 19\% & 1.22 & 0.24 & $[1.00,\,\ph{1}1.82]$ \\ 
 $\text{ABB}_{\text{bon}}$ & 100\% & 19\% & 1.33 & 0.43 & $[1.00,\,\ph{1}2.70]$ \\ 
\bottomrule
\end{tabular}}
}
\hfill
\subfloat[${\sf tol} = 10^{-6}$, starting point $10\,\bx_0$]{
\scalebox{0.7}{
\begin{tabular}{lrcccc}
\toprule
Stepsize & Solved & $\text{PR}=1$ & Avg & Sd & Range \\
 \midrule
 COT 11 & 100\% & 43\% & 1.14 & 0.24 & $[1.00,\,\ph{1}2.04]$ \\ 
 COT H1 & 100\% & 43\% & 1.18 & 0.27 & $[1.00,\,\ph{1}1.90]$ \\ 
 BB2 & 100\% & 38\% & 1.40 & 0.75 & $[1.00,\,\ph{1}3.93]$ \\ 
 COT 12 & 100\% & 33\% & 1.17 & 0.21 & $[1.00,\,\ph{1}1.66]$ \\ 
 COT 1H & 100\% & 33\% & 1.16 & 0.20 & $[1.00,\,\ph{1}1.74]$ \\ 
 ABB & 100\% & 29\% & 1.47 & 0.76 & $[1.00,\,\ph{1}3.98]$ \\ 
 COT 21 & 100\% & 29\% & 1.22 & 0.22 & $[1.00,\,\ph{1}1.62]$ \\ 
 ITER & 100\% & 29\% & 3.33 & 8.42 & $[1.00,\,39.94]$ \\ 
 $\text{ABB}_{\text{min}}$ & 100\% & 24\% & 1.46 & 0.80 & $[1.00,\,\ph{1}3.89]$ \\ 
 BB1 & 100\% & 19\% & 3.04 & 5.42 & $[1.00,\,21.66]$ \\ 
 IBB2 2.01 & 100\% & 19\% & 3.20 & 4.69 & $[1.00,\,21.73]$ \\ 
 IBB2 100 & 100\% & 19\% & 3.18 & 7.66 & $[1.00,\,36.43]$ \\ 
 $\text{ABB}_{\text{bon}}$ & 100\% & 14\% & 1.41 & 0.75 & $[1.00,\,\ph{1}3.89]$ \\ 
\bottomrule
\end{tabular}}
}
\end{table}

As in the quadratic case, Tables~\ref{tab:nfeunconstr} report some statistics on the performance ratios, based on the number of function evaluations, for ${\sf tol} = 10^{-6}$. As the performance profiles already suggested, when the starting point is $\bx_0$, the $\text{ABB}_{\text{min}}$ is the best stepsize in terms of proportion of solved problems, and problems solved at minimum cost. The BB2 step is the best method when the starting point is $5\,\bx_0$, immediately followed by $\text{ABB}_{\text{min}}$. Finally, for the problems with $10\,\bx_0$, the picture changes: COT~11 and COT~H1 are the best stepsizes, followed by BB2, but this time the range of BB2 is larger compared to the previous tables. As a consequence, $\text{ABB}_{\text{min}}$ and $\text{ABB}_{\text{bon}}$ solve a smaller proportion of problems at minimum cost. 

We have just shown that, as opposed to the quadratic case, there are situations where the TBB steps from the cotangent family show better performance than the ABB variants. In general, we observe that the stepsizes IBB2~100 and ITER are competitive with the BB1 step; IBB2~2.01 performs slightly worse, but this behavior can sometimes lead to better local optima (cf.~Remark~\ref{remark2}).  

\section{Conclusions} \label{sec:concl}
We have developed a harmonic framework for stepsize selection in gradient methods for unconstrained nonlinear optimization.
The harmonic steplength \eqref{hss} depending on targets $\tau_k$ is inspired by the harmonic Rayleigh--Ritz extraction for matrix eigenvalue problems.

The one-to-one relation between target and stepsize gives a general framework with new viewpoints and interpretations.
Compared to the eigenproblem context, where the target is commonly chosen inside the spectrum, in our situation we have studied both strategies with the target outside the spectrum and schemes that sometimes pick the target inside.
Targets on the negative real axis lead to stepsizes between BB2 and BB1. This yields connection with schemes such as \cite{dai2019family}.
We have analyzed and extended the popular ABB method. While the original ABB approach only allows a choice between two stepsizes based on a single parameter, we have introduced a new competitive family of stepsizes with tunable parameters, that enjoy the same key idea but are more flexible.
Additionally, we have considered new families of positives targets, leading to steplengths larger than the BB1 steps.
The use of harmonic stepsizes with target requires the same cost as the BB2 step, ABB, and ABB variants, and is only marginally more expensive than BB1.
The experiments suggest that both the cotangent family and the approaches with positive targets seem competitive \gf{with the well-known BB stepsizes and ABB; they compete also with the ABB variants \cite{frassoldati2008new, bonettini2009scaled} for generic unconstrained optimization problems.}

For an analysis of the new schemes, we have extended convergence results from Dai and Liao \cite{dai2002r} in Section~\ref{sec:conv}. In view of the TBB steps, instead of $\lambda_1 \le \alpha_k \le \lambda_n$, we have studied the more general setting $\xia \, \lambda_1 \le \alpha_k \le \xib \, \lambda_n$, particularly for $\frac12 < \xia \le 1$ and $\xib \ge 1$.

An R implementation of the methods described in this paper can be obtained from \href{https://github.com/gferrandi/tbbr}{github.com/gferrandi/tbbr}.

\medskip\noindent
{\bf Acknowledgment:}
The authors thank the referees and editor for their very useful comments.
This work has received funding from the European Union's Horizon 2020 research and innovation programme under the Marie Sklodowska-Curie grant agreement No 812912.

\medskip\noindent
{\bf Data Availability:} The data used during the current study are available in the SuiteSparse Matrix Collection repository, \href{https://sparse.tamu.edu/}{sparse.tamu.edu}.


\bibliography{harmref}

\end{document}